\documentclass[lettersize,journal]{IEEEtran}

\usepackage{amsmath,amssymb,amsfonts,mathtools}
\usepackage{amsthm}
\usepackage{bm}
\newtheorem{theorem}{Theorem}

\newtheorem{example}{Example}
\theoremstyle{remark}
\newtheorem{remark}{\textbf{Remark}}

\usepackage{graphicx}
\graphicspath{{Pictures/}}
\usepackage[caption=false,font=footnotesize,labelfont=sf,textfont=sf]{subfig}
\usepackage{stfloats}
\usepackage{xcolor}
\usepackage{url}
\usepackage{hyperref}
\hypersetup{colorlinks=true,linkcolor=blue,citecolor=blue,urlcolor=blue,breaklinks=true}

\usepackage{siunitx}

\usepackage{booktabs}
\usepackage{tabularx}
\usepackage{multirow}
\usepackage{makecell}
\usepackage{diagbox}
\usepackage{array}

\usepackage[ruled,vlined,linesnumbered]{algorithm2e}
\SetKwInput{KwIn}{Input}
\SetKwInput{KwOut}{Output}

\usepackage{comment}
\usepackage{float}

\usepackage{tikz}
\usetikzlibrary{arrows.meta, positioning}

\usetikzlibrary{shapes.geometric, arrows.meta, positioning, shadows, calc}

\usepackage{cite}
\SetKwBlock{Phase}{}{end}

\begin{document}

\title{Difference-of-Convex Regularization for Graph Learning by Differentiable Programming}

\author{Liping Tao and Chee Wei Tan, \textit{Senior Member, IEEE}

\thanks{Liping Tao (liping.tao@163.com) is with the College of Computing and Data Science, Nanyang Technological University, Singapore 639798.}

\thanks{Chee Wei Tan (cheewei.tan@ntu.edu.sg) is with the College of Computing and Data Science, Nanyang Technological University, Singapore 639798.}
}

% The paper headers
\markboth{Journal of \LaTeX\ Class Files,~Vol.~14, No.~8, August~2021}%
{Shell \MakeLowercase{\textit{et al.}}: A Sample Article Using IEEEtran.cls for IEEE Journals}

\maketitle

\begin{abstract}
Laplacian-regularized minimization is fundamental in signal processing and machine learning, but explicit construction of the dense graph Laplacian pseudoinverse can be computationally expensive. To address this issue, we consider the setting where the graph Laplacian is given and propose a Difference-of-Convex Regularizer (DCR) graph learning framework that learns a reusable approximation of the pseudoinverse action through regularized Maximum Likelihood Estimation (MLE) and shrinkage-regularized Convex--Concave Procedure (CCCP) iterations. For Laplacian-Regularized Nonnegative Least Squares (LR-NNLS), dual and Karush-Kuhn-Tucker (KKT) analysis motivates separating graph-dependent pseudoinverse learning from instance-specific optimization. The learned pseudoinverse is embedded into a differentiable primal optimization procedure as a graph-aware preconditioner for nonnegative reconstruction and can be reused across LR-NNLS instances sharing the same graph. We establish stability and fixed-point existence guarantees for the proposed pseudoinverse-learning iteration. Experiments across different graph topologies and scales demonstrate near-reference accuracy, competitive time-to-moderate-accuracy behavior, and effective cross-instance reuse across independently generated problem instances sharing the same graph.
\end{abstract}


\section{Introduction}
\label{sec:itr}
Laplacian-regularized minimization~\cite{tuck2019distributed} is a fundamental optimization paradigm that integrates task-specific objectives with graph-based smoothness priors, allowing learning and estimation problems to exploit relational structure among variables rather than treating them independently~\cite{yin2016laplacian}. A broad class of Laplacian-Regularized Minimization Problems (LRMPs) can be written as:
\begin{equation}
\min_{x\in\mathbb{R}^n}
\quad
f(x)+\frac{1}{2}x^\top Lx,
\label{eq:lrmp}
\end{equation}
where $f$ is convex and $L\succeq0$ is a symmetric singular graph Laplacian satisfying $L\mathbf{1}=0$. promoting smoothness by penalizing discrepancies across graph-connected components~\cite{tuck2019distributed}.

In many real-world applications, LRMPs arise with intrinsic constraints. The variables often represent physical or semantic quantities such as abundances, activation coefficients, importance scores, or resource allocations that are inherently nonnegative. This naturally leads to constrained formulations that enforce graph-induced smoothness while preserving feasibility and interpretability. A representative example is the Laplacian-Regularized Nonnegative Least-Squares (LR-NNLS) problem:
\begin{equation}
\label{prob:LR-NNLS}
\min_{x \succeq 0} \quad \frac{1}{2}\|Ax-b\|_2^2 + \frac{1}{2} x^\top L x,
\end{equation}
which is widely used in applications such as semi-supervised ranking on graphs \cite{gao2011semi} and graph-guided nonnegative coding \cite{el2013representing}. These ideas can also be extended to matrix problems, such as hyperspectral unmixing \cite{ammanouil2015graph}, graph-regularized nonnegative matrix factorization \cite{cai2010graph}, and Laplacian-regularized nonnegative representation \cite{zhao2020laplacian}. In all these cases, the nonnegativity constraint plays a key role in defining the solution space and influencing how the problem is solved.

\begin{figure}[!t]
\centering
\includegraphics[trim=1.2cm 0cm 0cm 0cm, clip, width=\linewidth]{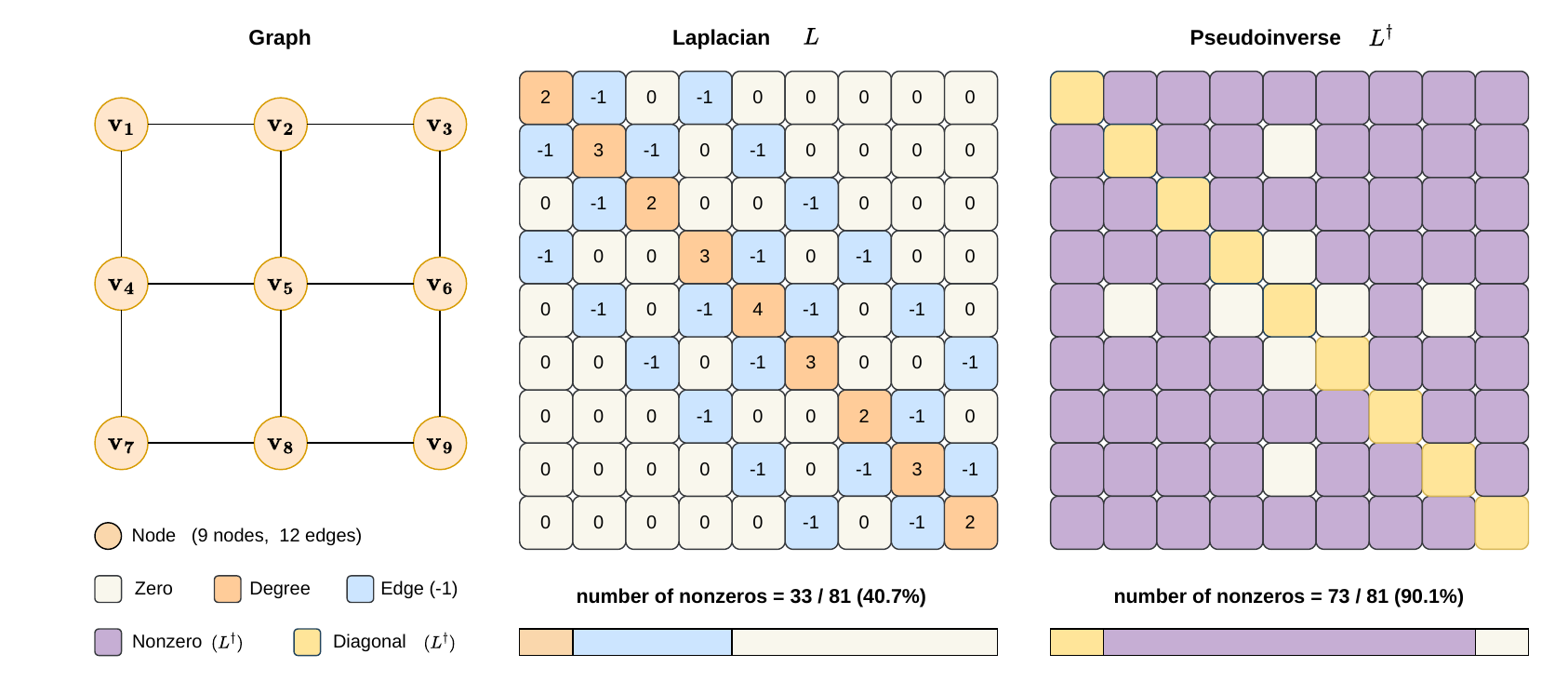}
\caption{While the graph Laplacian $L$ is sparse, its pseudoinverse $L^\dagger$ is generally dense, motivating reusable approximations that avoid repeated exact pseudoinverse construction.}
\label{fig:lrmp_graph}
\end{figure}

While the Laplacian term enforces smoothness across connected nodes, the singularity of $L$ implies that classical LRMP solutions are characterized through the Moore--Penrose pseudoinverse $L^{\dagger}$~\cite{boyd2004convex}. In what follows, we refer to $L^{\dagger}$ as the Laplacian pseudoinverse. As shown in Fig.~\ref{fig:lrmp_graph}, although $L$ is sparse, its pseudoinverse is generally dense, making explicit exact construction computationally expensive for large graphs~\cite{ranjan2014incremental}. This challenge is further exacerbated under additional constraints, such as nonnegativity in~\eqref{prob:LR-NNLS}, which preclude closed-form solutions and require iterative or implicit solvers. More broadly, the repeated construction and manipulation of inverse operators can be computationally demanding in modern learning systems~\cite{shivashankar2025maintainability}, where high-dimensional graph- or kernel-induced mappings rarely admit direct computation.

\textit{Motivation.}
Classical approaches to large-scale Laplacian problems include iterative solvers, preconditioning, and graph sparsification~\cite{spielmanlaplacian}. For constrained LRMPs, projected gradient methods provide simple and scalable solutions without requiring explicit pseudoinverse computation, but can converge slowly under spectral ill-conditioning. Spectral graph filtering and polynomial approximations, such as Chebyshev filtering~\cite{hammond2011wavelets, shuman2018distributed, cheng2025iterative}, provide another class of efficient approximations widely used in graph signal processing~\cite{choi2024graph, kunegis2009learning, kim2008learning}.
While these methods avoid explicit pseudoinverse computation, they rely on fixed spectral approximations rather than learning a reusable approximation of the Laplacian pseudoinverse, leaving room for further improvement under nonnegativity constraints.
This motivates a different question: \emph{given a graph Laplacian $L$, can we learn a reusable, graph-dependent approximation of the action of $L^\dagger$ without explicitly constructing the exact pseudoinverse, and subsequently embed this learned pseudoinverse into a constrained primal solver?} Such separation enables reuse across LR-NNLS instances sharing the same graph but differing in $A$ and $b$.

Graph learning~\cite{xia2026graph} broadly concerns learning graph structures or graph-induced operators. In this work, the graph Laplacian is given, and our focus is instead on learning a graph-dependent approximation of its pseudoinverse action. The associated regularized maximum-likelihood formulation exhibits a Difference-of-Convex (DC) structure. DC programming~\cite{tao1988duality,shen2016disciplined,yao2023globally,shen2023minimizing}, together with DC regularization~\cite{chen2025computing,li2021direct,tao1998dc}, provides a principled mechanism for incorporating structural regularization while retaining tractable iterative updates. \textit{Toland duality}~\cite{toland1978duality} further provides a useful nonconvex-duality perspective, while the Convex--Concave Procedure (CCCP)~\cite{yuille2003concave,lipp2016variations,yuille2001concave,tao2026accelerating} yields an efficient fixed-point implementation. These observations motivate our Difference-of-Convex Regularizer (DCR) framework.

\textit{Contributions.}
To overcome the above limitations, we propose a structured graph learning framework for Laplacian-regularized optimization that replaces \(L^{\dagger}\) with a learnable approximation via regularized Maximum Likelihood Estimation (MLE) and differentiable programming~\cite{blondel2024elements,tao2026learning}. Our main contributions are:
\begin{itemize}
\item We study the LR-NNLS problem~\eqref{prob:LR-NNLS}, derive its dual formulation to decouple $A$ from $L^{\dagger}$, and use KKT conditions to reveal the structure underlying primal solution recovery, motivating learnable and problem-adaptive pseudoinverse approximation.

\item We propose the DCR framework, which learns a reusable graph-aware preconditioner $\tilde{L}^{\dagger}$ from $L$ through regularized MLE and shrinkage-regularized CCCP iterations, and uses it in a differentiable primal optimization procedure to reconstruct the original LR-NNLS solution. Its \textit{novelty} lies in decoupling graph-dependent operator learning from instance-specific inference and unifying DC-regularized pseudoinverse approximation, reusable graph preconditioning, and nonnegativity-aware reconstruction in a single framework.

\item We establish stability and fixed-point properties of the proposed DCR algorithm. Experiments on diverse graph topologies and scales demonstrate near-reference accuracy, competitive time-to-moderate-accuracy, and effective cross-instance reuse. The source code is publicly available.\footnote{Source code is available at \url{https://github.com/convexsoft/LRMP}.}
\end{itemize}

\section{Related Work}
\label{sec:rw}
\subsection{Nonnegative LRMPs}

% \vspace{1mm}
\textit{Semi-Supervised Ranking.} 
A closely related line of work is the semi-supervised ranking framework of~\cite{gao2011semi}, which studies large-scale graph-based ranking with rich node and edge metadata. The framework formulates ranking as an optimization problem over a nonnegative score vector, combining graph-based smoothness induced by random walks or graph Laplacians with explicit pairwise preference supervision. As a special case, the Laplacian rank formulation minimizes a quadratic objective defined by a directed graph Laplacian, subject to margin-based pairwise preference constraints and a nonnegativity requirement on the ranking vector. This yields a supervised Laplacian-regularized optimization problem that extends PageRank-style models and is closely related to the class of nonnegative LRMPs of our work.

% \vspace{1mm}
\textit{Graph-Aware Dictionary Learning.}
A related example appears in the badge-based document representation framework of~\cite{el2013representing}, which encodes each document as a nonnegative sparse coefficient vector over a learned badge dictionary. The coding stage solves a nonnegative least-squares/lasso-type problem with a graph-guided fused regularizer on a badge relation graph \cite{kim2009multivariate, chen2012smoothing}, encouraging correlated badges to take similar coefficients via weighted pairwise differences. This yields a structured nonnegative optimization that combines reconstruction, sparsity, and graph-based smoothness. Although the regularizer is not a quadratic Laplacian form, it is closely related to nonnegative LRMPs in that it integrates nonnegativity with graph-structured regularization for coherent representations.

% \vspace{1mm}
\textit{Hyperspectral Data Unmixing.}
Graph Laplacian regularization has also been studied in nonnegative matrix-valued optimization problems beyond ranking. In hyperspectral data unmixing, Ammanouil et al.~\cite{ammanouil2015graph} proposed a Laplacian-regularized convex formulation to estimate abundance maps under nonnegativity and sum-to-one constraints. The model combines a quadratic data-fidelity term with a graph-based smoothness regularizer, promoting spatial and spectral consistency over a similarity graph. This formulation can also be viewed as a nonnegative Laplacian-regularized least-squares problem, closely related to the nonnegative LRMPs considered in this work, differing mainly in its matrix-valued variables and application setting.

% \vspace{1mm}
\textit{Nonnegative Matrix Factorization.}
Nonnegative Laplacian-regularized problems also arise in nonnegative matrix factorization. Graph regularized nonnegative matrix factorization \cite{cai2010graph} extends standard nonnegative matrix factorization by incorporating a Laplacian term to preserve data geometry, leading to a quadratic reconstruction loss with a smoothness penalty under elementwise nonnegativity constraints. It is a nonnegative Laplacian-regularized problem where the Laplacian enforces graph smoothness and feasibility is maintained via constrained updates. Although designed for matrix-valued representations in clustering and representation learning, its core structure aligns closely with the nonnegative Laplacian-regularized formulations.

% \vspace{1mm}
\textit{Nonnegative Self-Representation.}
Nonnegative Laplacian-regularized formulations have also been studied in self-representation models for clustering and dimensionality reduction. For instance, Laplacian-regularized nonnegative representation~\cite{zhao2020laplacian} learns a nonnegative coefficient matrix under a self-expressive framework, enhanced with graph-based smoothness regularization. This formulation promotes interpretability via nonnegativity while enforcing smoothness over a similarity graph. It can be viewed as a matrix-valued extension of nonnegative LRMPs, closely related to LR-NNLS, but differing in its self-representation structure and loss design.

%Graph learning studies how to infer graph structures or graph-induced operators from data \cite{xia2026graph, xia2021graph}.

\subsection{Graph Learning}
Graph learning spans diverse directions, including graph neural networks stability, generalization, and heterophily~\cite{yang2026deeper,zheng2026graph}. Our work instead focuses on learning a reusable pseudoinverse approximation for a given graph Laplacian.

\textit{Laplacian / Graph Structure Learning.}
A major line of work focuses on learning the graph Laplacian. Pavez et al.~\cite{pavez2016generalized} formulate this task as maximum likelihood estimation of a precision matrix under generalized Laplacian constraints, enabling flexible Gaussian Markov random field modeling. Data-driven approaches further construct graphs via structured representations, such as the nonnegative low-rank and sparse framework in \cite{zhuang2012non}, which captures global and local structures, and graph refinement methods like \cite{luo2012forging}, which project an initial similarity matrix onto a nonnegative, low-rank, and positive semidefinite set to obtain a denoised graph.

\textit{Laplacian Pseudoinverse Learning.}
Complementary to structure learning, recent works aim to directly learn or approximate the Laplacian pseudoinverse, which encodes global connectivity and diffusion behavior. Alfke et al.~\cite{alfke2021pseudoinverse} learn the pseudoinverse implicitly from graph signals via algebraic constraints, and Zhu et al.~\cite{zhu2024efficient} develop structured approximations for large graphs, while Lu et al.~\cite{lu2025diagonal} improve efficiency and stability via diagonal preconditioning. Our work focus on learning the pseudoinverse of a given Laplacian matrix using regularized MLE and CCCP.

\section{Laplacian-Regularized Nonnegative Least Squares}
\label{sec:system-model}
Given a data matrix $A \in \mathbb{R}^{m\times n}$ and a vector $b \in \mathbb{R}^m$, the objective is to obtain a nonnegative decision variable $x \in \mathbb{R}^n$ by solving:
\begin{equation}
\label{prob:Learning_LRMP_primal}
\min_{x \succeq 0} \quad \frac{1}{2}\|Ax-b\|_2^2 + \frac{1}{2} x^\top L x,
\end{equation}
where $L \in \mathbb{R}^{n\times n}$ is given and denotes the graph Laplacian associated with an undirected weighted graph $(\mathcal{V}, \mathcal{E})$, where $w_{ij} \ge 0$ denotes the weight of edge $(i,j)$. The quadratic penalty admits:
\begin{equation}
\label{eq:LR-NNLS-Laplacian}
x^\top L x = \sum_{(i,j)\in\mathcal{E}} w_{ij}(x_i - x_j)^2.
\end{equation}

\textit{Spectral Challenges.}
Assuming $L^{\dagger}$ is known, solving~\eqref{prob:Learning_LRMP_primal} remains challenging due to spectral ill-conditioning. As shown in Fig.~\ref{fig:Spectrum_and_Convergence_of_LRNNLS}, using a chain graph Laplacian ($n=500$, $m=25$, $w=0.01$, $\kappa(A^\top A+L)\approx 3.64\times10^4$) as a representative ill-conditioned instance, although the objective gap decreases rapidly for the first-order methods such as projected gradient descent, $\|G_\eta(x)\|$ remains above $10^{-3}$ throughout 3000 iterations, indicating slow convergence to stationarity. The nonnegativity constraint further complicates the optimization.

\begin{figure}
    \centering
    \includegraphics[width=\columnwidth]{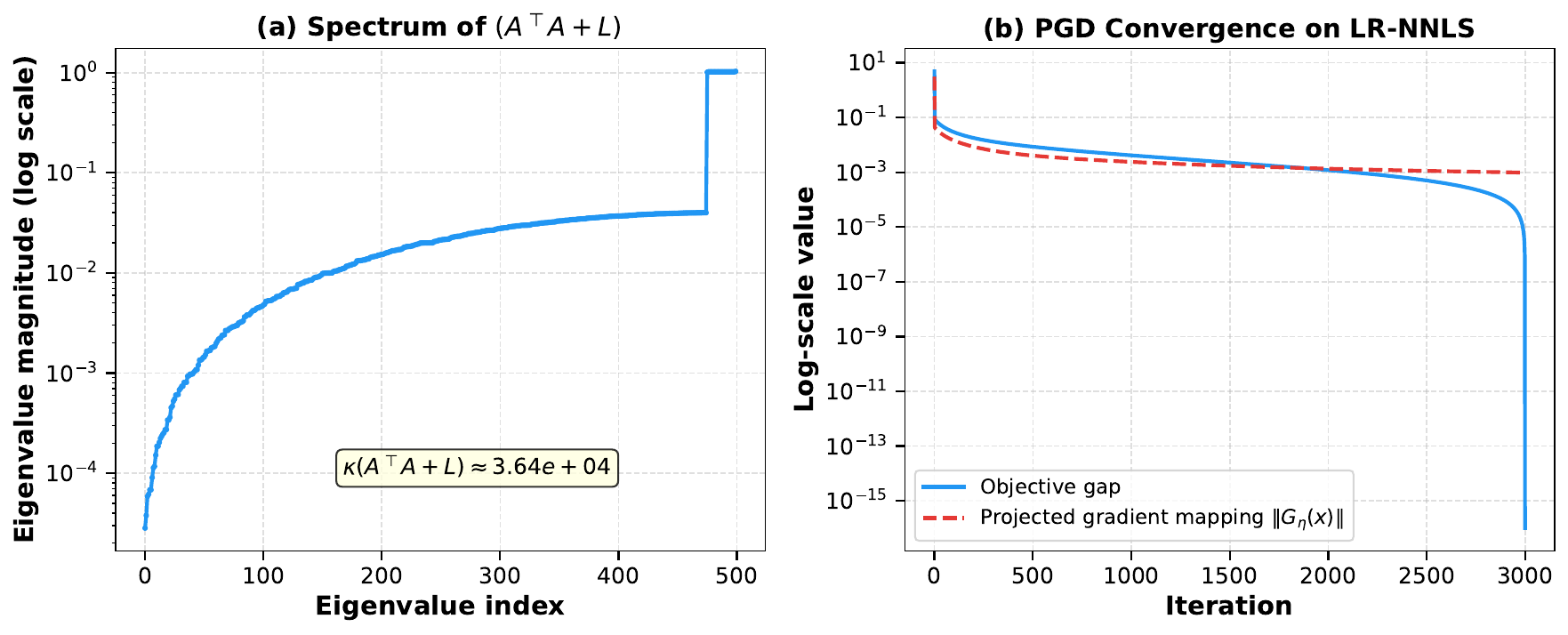}
    \caption{Spectral and convergence properties of the LR-NNLS problem~\eqref{prob:Learning_LRMP_primal} ($n=500$, $m=25$, chain graph, $w=0.01$, $\kappa(A^\top A+L)\approx3.64\times10^4$). (a) Eigenvalue spectrum of $A^\top A+L$, showing severe ill-conditioning. (b) PGD convergence, where the objective gap decreases rapidly while the projected gradient mapping norm $\|G_\eta(x)\|$ remains above $10^{-3}$ after 3000 iterations.}
    \label{fig:Spectrum_and_Convergence_of_LRNNLS}
\end{figure}

\subsection{Dual Perspective and KKT Conditions}
Since the complexity of the matrix $A$ directly impacts the difficulty of approximating $(A^\top A + L)^{\dagger}$, we adopt a dual formulation to decouple the Laplacian pseudoinverse from the specific realization of $A$. Introducing the auxiliary variable $y = Ax - b$, dual variables $\lambda \in \mathbb{R}^m$ associated with the constraint $y = Ax - b$, and $\mu \succeq 0$ associated with the constraint $x \succeq 0$, the resulting Lagrangian is:
\begin{equation}
\label{eq:Learning_LRMP_lagrangian}
\mathcal{L}(x,y;\lambda,\mu)
=
\tfrac12\|y\|_2^2
+ \tfrac12 x^\top L x
+ \lambda^\top(Ax - b - y)
- \mu^\top x .
\end{equation}

%\vspace{1mm}
\textit{KKT Conditions.}
Stationarity of $\mathcal{L}$ with respect to $y$ and $x$ yields $y^\star=\lambda^\star$ and $Lx^\star+A^\top\lambda^\star-\mu^\star=0$, respectively. Together with primal and dual feasibility and complementary slackness, the KKT conditions are:
\begin{align}
x^\star \succeq 0, 
\quad
\lambda^\star = y^\star = Ax^\star-b,
\quad
\mu^\star \succeq 0,
\label{eq:feasibility}\\
(\mu^\star)^\top x^\star = 0,
\label{eq:complementary-slackness}\\
Lx^\star + A^\top\lambda^\star-\mu^\star = 0.
\label{eq:stationarity}
\end{align}

\begin{theorem}
Suppose that $L$ is the Laplacian of a connected graph and that $(x^\star,\lambda^\star,\mu^\star)$ satisfies the KKT conditions \eqref{eq:feasibility}--\eqref{eq:stationarity}. Then the primal solution admits:
\begin{equation}
x^\star
=
L^\dagger(\mu^\star-A^\top\lambda^\star)
+
c^\star\mathbf{1},
\label{eq:primal-solution-recovery-from-kkt}
\end{equation}
where $c^\star\mathbf{1}$ represents the null-space component of $L$.
\end{theorem}

\begin{proof}
From the stationarity condition \eqref{eq:stationarity},
\begin{equation}
Lx^\star
=
\mu^\star-A^\top\lambda^\star.
\label{eq:Lx-equation}
\end{equation}

To characterize the solution structure of \eqref{eq:Lx-equation}, we recall fundamental properties of linear algebra \cite{boyd2004convex, strang2022introduction}. For any matrix $G \in \mathbb{R}^{m \times n}$ and equation $Gx=g$, the range space is $\mathcal{R}(G) = \{Gx \mid x \in \mathbb{R}^n\}$ and the null space is $\mathcal{N}(G) = \{x \mid Gx=0\}$. The general solution, when it exists, takes the form $x=x_p+x_n$, where $Gx_p=g$ is a particular solution and $x_n \in \mathcal{N}(G)$ is an arbitrary null space component. Specifically, If $G$ is invertible, $\mathcal{N}(G)=\{0\}$ and the unique solution is $x=G^{-1}g$. If $G$ is singular but $g \in \mathcal{R}(G)$, the general solution is $x=G^{\dagger}g+x_n$ where $x_n \in \mathcal{N}(G)$ and $G^{\dagger}$ denotes the Moore--Penrose pseudoinverse.

The graph Laplacian $L$ is symmetric positive semidefinite with $L\mathbf{1}=0$, where $\mathbf{1} \in \mathbb{R}^n$ denotes the all-ones vector. For a connected graph, $\mathcal{N}(L)=\text{span}\{\mathbf{1}\}$ is one-dimensional, where “span” describes all the linear combinations of a set of vectors. Therefore, the right-hand side of \eqref{eq:Lx-equation} must lie in $\mathcal{R}(L)$, which is guaranteed by the complementary slackness of the KKT conditions. Thus, the general solution of \eqref{eq:Lx-equation} is $x^\star = L^{\dagger}(\mu^\star - A^\top \lambda^\star) + c^\star\mathbf{1}$.
\end{proof}

\begin{remark}
Let $\ell^\star=L^\dagger(\mu^\star-A^\top\lambda^\star)$. The null-space coefficient $c^\star$ in \eqref{eq:primal-solution-recovery-from-kkt} is generally not arbitrary. Since $L\mathbf{1}=0$, the Laplacian regularization is invariant along $\ell^\star+c\mathbf{1}$, while nonnegativity requires $c\geq-\min_i\ell_i^\star$. Therefore,
\begin{equation}
c^\star\in
\arg\min_{c\geq-\min_i\ell_i^\star}
\frac{1}{2}\|A\ell^\star+cA\mathbf{1}-b\|_2^2.
\label{eq:c-scalar-problem}
\end{equation}
If $A\mathbf{1}\neq0$, the unique solution is:
\begin{equation}
c^\star=
\max\left\{
-\min_i\ell_i^\star,\,
\frac{(A\mathbf{1})^\top(b-A\ell^\star)}
{\|A\mathbf{1}\|_2^2}
\right\}.
\label{eq:c-closed-form}
\end{equation}
If $A\mathbf{1}=0$, any feasible $c$ yields the same objective value.
\end{remark}

\begin{example}
For an LR-NNLS instance with $m=50$ and $n=30$, a reference KKT point is obtained using CVXPY~\cite{diamond2016cvxpy}. Here, $\|A\mathbf{1}\|_2=42.64\neq0$, and~\eqref{eq:c-closed-form} gives $c^\star=6.21\times10^{-2}$. The resulting reconstruction achieves a relative solution error of $2.09\times10^{-6}$, whereas setting $c=0$ yields an error of $5.26\times10^{-1}$ and violates nonnegativity. This confirms that the null-space component in~\eqref{eq:primal-solution-recovery-from-kkt} is generally nontrivial.
\end{example}

\begin{figure*}[!t]
\centering
\includegraphics[trim=0cm 0.4cm 0cm 0.4cm, clip, width=\linewidth]{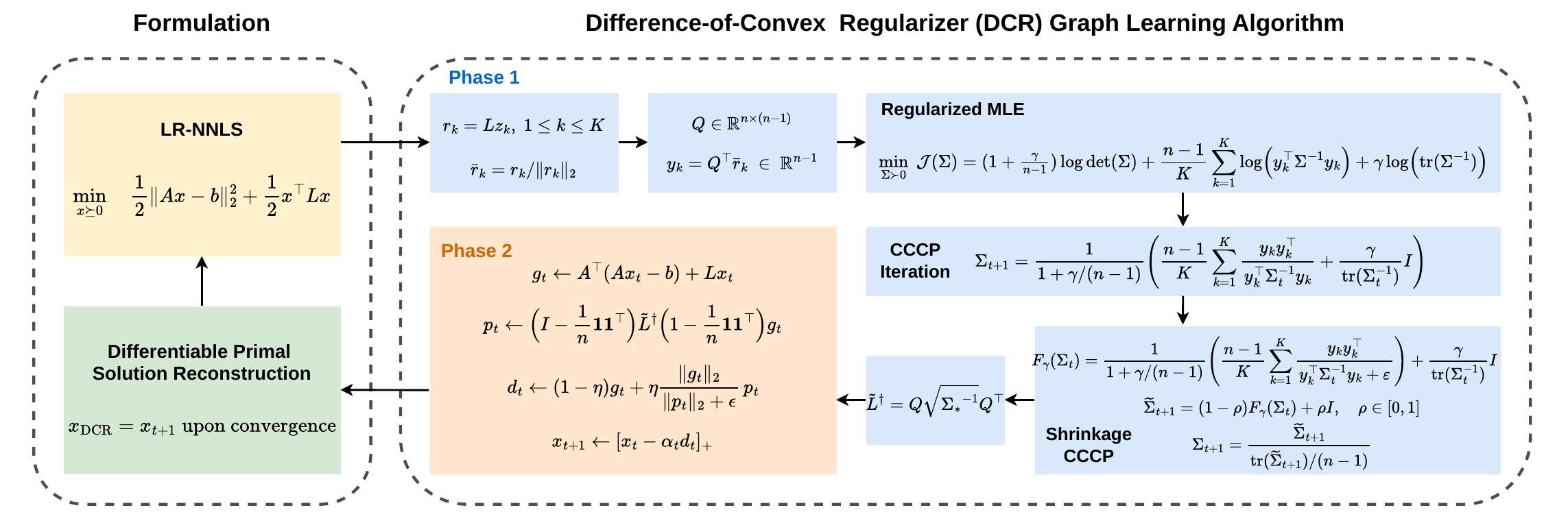}
\caption{DCR solves \eqref{prob:Learning_LRMP_primal} in two phases. Phase~1 learns a graph-dependent approximation $\tilde{L}^{\dagger}$ of the Laplacian pseudoinverse through regularized MLE and shrinkage-based CCCP iterations. Phase~2 incorporates the learned operator into a differentiable programming pipeline for primal reconstruction, where it serves as a graph-aware preconditioner in projected gradient updates.}
\label{fig:DCR-flow}
\end{figure*}

\textit{Objective.}
The KKT characterization above reveals the structural role of $L^{\dagger}$ in the primal solution. However, computing $L^{\dagger}$ via eigendecomposition or direct matrix factorization is computationally prohibitive for large-scale systems. This motivates learning an approximation of $L^{\dagger}$ on $\mathcal{R}(L)$, which is subsequently used as a graph-aware preconditioner for direct primal optimization.

\section{Laplacian Estimation}
\label{sec:alg-DCR}
Since $L^{\dagger}$ is unavailable, we propose the DCR framework to learn a reusable approximation and use it for solving~\eqref{prob:Learning_LRMP_primal}. As shown in Algorithm~\ref{alg:dlpls-core} and Fig.~\ref{fig:DCR-flow}, DCR consists of two phases. First, given $L$, DCR learns $\widetilde{L}^{\dagger}\approx L^{\dagger}$ by solving a regularized MLE problem via CCCP~\cite{yuille2003concave,lipp2016variations,yuille2001concave,tao2026accelerating}. Since $\widetilde{L}^{\dagger}$ depends only on $L$, it can be reused across LR-NNLS instances sharing the same graph. Second, $\widetilde{L}^{\dagger}$ is used as a graph-aware preconditioner within projected gradient updates to solve~\eqref{prob:Learning_LRMP_primal} directly in the primal space.

\subsection{Regularized MLE for Laplacian Pseudoinverse Learning}
\label{subsec:phase1}
Inspired by Tyler's estimator \cite{tyler1987distribution, sun2014regularized, soloveychik2015tyler, palomar2025portfolio}, this phase formulates a scale-invariant regularized MLE problem \cite{fisher1922mathematical} to learn $\tilde{L}^{\dagger} \in \mathbb{R}^{n \times n}$ such that $\tilde{L}^{\dagger} \approx L^{\dagger}$ on $\mathcal R(L)$.

%\vspace{1mm}
\subsubsection{Tyler's Estimator}
Tyler's estimator stems from the Angular Central Gaussian (ACG) model, which describes scaled Gaussian vectors $x_i=\sqrt{\tau_i}\xi_i$ with $\xi_i\sim\mathcal N(0,\Sigma)$, $\tau_i>0$, and $\Sigma\in\mathbb R^{p\times p}$. Given $N_s$ samples $\{x_i\}_{i=1}^{N_s}$, the sample covariance is $\widehat S = \frac{1}{N_s} \sum_{i=1}^{N_s}x_ix_i^\top$. Tyler's estimator \cite{sun2014regularized} satisfies:
\begin{equation}
\Sigma
=
\frac{p}{N_s}
\sum_{i=1}^{N_s}
\frac{x_ix_i^\top}
{x_i^\top\Sigma^{-1}x_i},
\label{eq:tyler-fixed-point}
\end{equation}
which is the stationarity condition of the ACG negative log-likelihood:
\begin{equation}
\ell(\Sigma)
=
\log\det(\Sigma)
+
\frac{p}{N_s}
\sum_{i=1}^{N_s}
\log\left(
x_i^\top\Sigma^{-1}x_i
\right).
\end{equation}
Due to the scale invariance $\ell(c\Sigma)=\ell(\Sigma)$ for $c>0$, a normalization such as $\operatorname{tr}(\Sigma)=p$ is typically imposed to ensure identifiability while preserving directional information.

%\vspace{1mm}
\subsubsection{Spectral Exploration}
In the absence of observed graph signals~\cite{pavez2016generalized}, we generate random directions to explore the spectrum of the Laplacian. For a connected undirected graph, $\mathcal{N}(L)=\operatorname{span}\{\mathbf{1}\}$. Since $L$ is symmetric, $\mathcal{R}(L) = \mathcal{N}(L)^\perp = \mathbf{1}^\perp$, which has dimension $n-1$. Therefore, let $Q\in\mathbb{R}^{n\times(n-1)}$ be any orthonormal basis of $\mathbf{1}^\perp$, satisfying $Q^\top Q = I_{n-1}$ and $Q^\top \mathbf{1}=0$. Such a basis can be constructed directly without eigendecomposition of $L$.
To cover all spectral directions, we sample isotropic Gaussian vectors $z_k\sim\mathcal{N}(0,I_n)$ and construct:
\begin{align}
r_k = Lz_k,\quad
\bar r_k = \frac{r_k}{\|r_k\|_2},
\quad k=1,\dots,K.
\end{align}

Then $r_k\in\mathcal{R}(L)=\mathbf{1}^\perp$, and hence $\bar r_k\in\mathcal{R}(L)$. With sufficiently large $K$, the set $\{r_k\}$ spans $\mathcal{R}(L)$ with high probability. The normalization removes magnitude variability while preserving directional information. Following the Tyler framework \cite{tyler1987distribution, soloveychik2015tyler}, each normalized vector admits intrinsic coordinates $y_k = Q^\top \bar r_k \in \mathbb{R}^{n-1}$, so that $\bar r_k = Qy_k$. The subsequent estimation is invariant to the particular choice of orthonormal basis $Q$.

%\vspace{1mm}

\subsubsection{Regularized MLE on $\mathcal R(L)$}
To estimate $\tilde{L}^{\dagger}$ on $\mathcal R(L)$, we consider a regularized MLE formulation~\cite{palomar2025portfolio}:
\begin{align}
\min_{\Sigma\succ0}\;
J(\Sigma)
=\;
&(1+\tfrac{\gamma}{n-1})\log\det(\Sigma)
+
\frac{n-1}{K}\sum_{k=1}^K
\log\!\big(y_k^\top \Sigma^{-1} y_k\big) \notag \\
&+
\gamma\,\mathcal{R}(\Sigma^{-1}),
\end{align}
where $\mathcal{R}(\cdot)$ is a spectral regularizer acting on the eigenvalues of $\Sigma^{-1}$. Such regularizers are commonly used to control the spectrum of $\Sigma$, in particular to prevent eigenvalue degeneracy, i.e., eigenvalues approaching zero. Representative choices are summarized in Table~\ref{tab:regularizers} (see Appendix~\ref{sec:appendix-cccp-regularizers} for details). In this work, we adopt the scale-invariant regularizer $\mathcal{R}(\Sigma^{-1}) = \log\!\big(\mathrm{tr}(\Sigma^{-1})\big)$ to control the spectrum while preserving scale invariance. The resulting problem is:
\begin{align}
\label{prob:lrmp-regu-mle-Sigma-objective}
\min_{\Sigma\succ0}\;
J(\Sigma)
=\;
&(1+\tfrac{\gamma}{n-1})\log\det(\Sigma)
+
\frac{n-1}{K}\sum_{k=1}^K
\log\!\big(y_k^\top \Sigma^{-1} y_k\big)
\notag\\
&+
\gamma\log\!\big(\mathrm{tr}(\Sigma^{-1})\big),
\end{align}
with $\gamma\ge0$. To ensure identifiability under scale invariance, we impose $\mathrm{tr}(\Sigma)=n-1$.

Since $r_k = L z_k$ with isotropic $z_k \sim \mathcal N(0,I_n)$, we have $\mathbb E[r_k r_k^\top]=L^2$. Although normalization removes scale variability, the directional structure remains aligned with $L^2$ on $\mathcal R(L)$. Hence, the optimal solution $\Sigma_\star$ of \eqref{prob:lrmp-regu-mle-Sigma-objective} shares eigenvectors with $L$ and is spectrally aligned with $L^2$, implying $\Sigma_\star^{-1}\propto (L^2)^\dagger$ on $\mathcal R(L)$. Let ${\tilde{L}^{\dagger}}_{\mathcal R(L)} = \sqrt{\Sigma_\star^{-1}}$, defined on $\mathcal R(L)$ and sharing eigenvectors with $L$. Then $\tilde{L}^{\dagger}$ is given by $\tilde{L}^{\dagger} = Q\sqrt{\Sigma_\star^{-1}}Q^\top$, which serves as a learned approximation of $L^\dagger$.

\subsection{CCCP Iteration}
For the regularized objective~\eqref{prob:lrmp-regu-mle-Sigma-objective}, define $S_k \triangleq y_k y_k^\top$ and introduce $X \triangleq \Sigma^{-1} \succ 0$, then the problem becomes:
\begin{align}
\label{prob:regu-mle-tyler-X-form}
\min_{X\succ 0}\;
F(X)
=
&-(1+\frac{\gamma}{n-1})\log\det(X) \notag \\
&+
\frac{n-1}{K}\sum_{k=1}^K 
\log\!\big(\mathrm{tr}(X S_k)\big)
+
\gamma \log\!\big(\mathrm{tr}(X)\big).
\end{align}
Write $F(X)=f(X)+g(X)$ with:
\begin{align}
f(X) &=-\;(1+\frac{\gamma}{n-1})\log\det(X), \label{eq:regu-mle-cccp-fx}
\\
g(X) &=
\frac{n-1}{K}\sum_{k=1}^K \log(\mathrm{tr}(X S_k))
+
\gamma \log(\mathrm{tr}(X)), \label{eq:regu-mle-cccp-gx}
\end{align}
where $f(X)$ is convex and $g(X)$ is concave. 

Defining $h(X)=-g(X)$, we can rewrite \eqref{prob:regu-mle-tyler-X-form} to the Difference-of-Convex (DC) programming \cite{tao1988duality, shen2016disciplined, lipp2016variations, yao2023globally, shen2023minimizing}:
\begin{align}
\label{eq:regu-mle-cccp-dcp}
\min_{X\succ 0}\;F(X)=f(X)-h(X).
\end{align}
According to the definition of the convex conjugate, $h(X)$ admits the representation:
\begin{align}
h(X)=\sup_Y \{\langle X,Y\rangle-h^*(Y)\}.
\end{align}
Substituting this representation to \eqref{eq:regu-mle-cccp-dcp} yields:
\begin{align}
f(X)-h(X)
=
\inf_Y \{f(X)-\langle X,Y\rangle+h^*(Y)\},
\end{align}
which leads to the associated DC dual problem \cite{tao1988duality}:
\begin{align}
\min_Y\; h^*(Y)-f^*(Y).
\end{align}
At an optimal solution $X^\star$, the DC optimality condition \cite{lipp2016variations, toland1978duality} requires $\partial h(X^\star)\subset \partial f(X^\star)$, which reduces to the stationarity condition:
\begin{align}
\label{eq:regu-mle-cccp-dcp-stationarity}
\nabla f(X^\star)=\nabla h(X^\star),
\end{align}
when $f$ and $h$ are differentiable.

\begin{table}[!t]
\centering
\caption{Representative regularizers for Laplacian estimation.}
\label{tab:regularizers}
\begin{tabular}{p{1.5cm}p{1.5cm}p{4.7cm}}%{lcc}
\toprule
\textbf{Regularizer} & \textbf{Spectrum} & \textbf{Remarks} \\
\midrule
$\mathrm{tr}(\Sigma^{-1})$
& $\sum_i \lambda_i^{-1}$
& Strong penalty on small eigenvalues \\

$\mathrm{tr}(\Sigma^{-1/2})$
& $\sum_i \lambda_i^{-1/2}$
& Moderate spectral shrinkage \\

$\log(\mathrm{tr}(\Sigma^{-1}))$
& $\log(\sum_i \lambda_i^{-1})$
& Scale-invariant and mild regularization \\

$-\log\det(\Sigma)$
& $-\sum_i \log\lambda_i$
& Log-barrier against eigenvalue collapse \\

$\|\Sigma^{-1}\|_*$
& $\sum_i \lambda_i^{-1}$
& Equivalent to $\mathrm{tr}(\Sigma^{-1})$ for PSD matrices \\

$\|\Sigma^{-1/2}\|_*$
& $\sum_i \lambda_i^{-1/2}$
& Nuclear-norm form of spectral shrinkage \\
\bottomrule
\end{tabular}
\end{table}

The CCCP method \cite{yuille2003concave, lipp2016variations, yuille2001concave} provides an iterative scheme for the stationarity condition \eqref{eq:regu-mle-cccp-dcp-stationarity}. Specifically, linearizing the concave part $g(X)$ in \eqref{eq:regu-mle-cccp-gx} at $X_t$ yields the following convex subproblem:
\begin{equation}
\label{eq:cccp-subproblem}
X_{t+1}
\in
\arg\min_{X\succ 0}
f(X)
+
\langle \nabla g(X_t), X \rangle .
\end{equation}
The optimality condition gives $\nabla f(X_{t+1}) = - \nabla g(X_t)$.
Since:
\begin{align}
\nabla f(X) &= -(1+\frac{\gamma}{n-1})X^{-1}, \\
\nabla g(X) &= \frac{n-1}{K}\sum_{k=1}^K \frac{S_k}{\mathrm{tr}(X S_k)} + \frac{\gamma}{\mathrm{tr}(X)}I,
\end{align}
we obtain:
\begin{equation}
\label{eq:cccp-x-iteration}
(1+\frac{\gamma}{n-1})X_{t+1}^{-1}
=
\frac{n-1}{K}\sum_{k=1}^K
\frac{y_k y_k^\top}{y_k^\top X_t y_k}
+
\frac{\gamma}{\mathrm{tr}(X_t)}I .
\end{equation}
Using $\Sigma_t=X_t^{-1}$ yields the regularized fixed-point iteration:
\begin{equation}
\label{eq:cccp-regu-mle-sigma-iteration}
\Sigma_{t+1}
=
\frac{1}{1+\gamma/({n-1})}
\Big(
\frac{n-1}{K}\sum_{k=1}^K
\frac{y_k y_k^\top}{y_k^\top \Sigma_t^{-1} y_k}
+
\frac{\gamma}{\mathrm{tr}(\Sigma_t^{-1})}I
\Big).
\end{equation}
When $\gamma=0$, the iteration reduces to~\eqref{eq:tyler-fixed-point}.

By the standard CCCP descent property~\cite{yuille2003concave}, i.e., $F(X_{t+1}) \le F(X_t)$, we equivalently obtain:
\begin{equation}
\label{eq:mono-sigma}
\mathcal J(\Sigma_{t+1})
\le
\mathcal J(\Sigma_t),
\qquad \forall t \ge 0.
\end{equation}

We solved \eqref{prob:lrmp-regu-mle-Sigma-objective} using three approaches: the CCCP iteration \eqref{eq:cccp-regu-mle-sigma-iteration}, a convex solver-based method, and the Disciplined Convex–Concave Programming (DCCP) package \cite{shen2016disciplined}. All three methods converge to the same stationary point, providing numerical validation of the CCCP iteration.

\begin{algorithm}[!t]
\caption{\textbf{DCR Algorithm}}
\label{alg:dlpls-core}
\KwIn{$A$, $b$, $L$.}
\KwOut{$\widetilde{L}^{\dagger}$ and $x_{\mathrm{DCR}}$.}

\textbf{Phase 1: Pseudoinverse Approximation Learning} \\

Construct an orthonormal basis $Q\in\mathbb{R}^{n\times(n-1)}$ of
$\mathbf{1}^\perp=\mathcal{R}(L)$ without eigendecomposition of $L$;
sample $z_k\sim\mathcal{N}(0,I_n)$, set
$r_k=Lz_k$, $\bar r_k=r_k/\|r_k\|_2$, and
$y_k=Q^\top\bar r_k \in \mathbb{R}^{n-1}$, $k=1,\dots,K$;

Initialize $\Sigma_1\gets I_{n-1}$;

\For{$t=1,\dots,T_{\mathrm{fp}}$}{
    $\displaystyle
    S \gets \frac{n-1}{K}\sum_{k=1}^K
    \frac{y_k y_k^\top}{y_k^\top\Sigma_t^{-1}y_k+\varepsilon};
    $
    
    $\displaystyle
    F_\gamma(\Sigma_t)\gets
    \frac{1}{1+\gamma/(n-1)}
    \Big(
    S+\frac{\gamma}{\mathrm{tr}(\Sigma_t^{-1})}I
    \Big);
    $
    
    $\widetilde\Sigma_{t+1}\gets(1-\rho)F_\gamma(\Sigma_t)+\rho I$;
    
    $\displaystyle
    \Sigma_{t+1}\gets
    \frac{\widetilde\Sigma_{t+1}}
    {\mathrm{tr}(\widetilde\Sigma_{t+1})/(n-1)};
    $
}

Compute $\displaystyle \widetilde{L}^{\dagger}\gets Q\sqrt{\Sigma_\star^{-1}}Q^\top$;

\textbf{Phase 2: Differentiable Primal Reconstruction} \\

Initialize $x_0\gets0$ and
$P_\perp\gets I-\frac{1}{n}\mathbf{1}\mathbf{1}^\top$;

\For{$t=0,\dots,T_2-1$}{
    $\displaystyle g_t\gets A^\top(Ax_t-b)+Lx_t$, $\;$ $\displaystyle g_{t,\perp}\gets P_\perp g_t$;
    
    %$\displaystyle g_{t,\perp}\gets P_\perp g_t$;
    
    $\displaystyle p_t\gets P_\perp\widetilde{L}^{\dagger}g_{t,\perp}$, $\;$ $\displaystyle \widehat p_t\gets\frac{\|g_t\|_2}{\|p_t\|_2+\epsilon}p_t$;
    
    % $\displaystyle \widehat p_t\gets\frac{\|g_t\|_2}{\|p_t\|_2+\epsilon}p_t$;
    
    $\displaystyle d_t\gets(1-\eta)g_t+\eta\widehat p_t$;

    Select $\alpha_t$ by projected backtracking line search;
    
    $\displaystyle x_{t+1}\gets[x_t-\alpha_t d_t]_+$;
    
    \If{the preconditioned step is not accepted}{
        $d_t\gets g_t$ and recompute the projected update;
    }
    
    \If{$\displaystyle
    \frac{\|x_{t+1}-[x_{t+1}-g(x_{t+1})]_+\|_2}{\sqrt n}
    \le\tau_{\mathrm{pg}}$}{
        \textbf{break};
    }
}

\Return{$\widetilde{L}^{\dagger}$, $x_{\mathrm{DCR}}\gets x_{t+1}$.}
\end{algorithm}

%the Frank--Wolfe (FW) algorithm \cite{frank1956algorithm, lacoste2016convergence, millan2023frank}

\begin{remark}
\label{rem:tyler-fw}
Alternatively, problem \eqref{prob:regu-mle-tyler-X-form} can also be solved using the Frank--Wolfe (FW) algorithm \cite{frank1956algorithm, millan2023frank}, originally developed for minimizing convex quadratic functions over polytopes \cite{yurtsever2022cccp}. Specifically, recall the DC objective \eqref{eq:regu-mle-cccp-dcp}:
\begin{equation}
\min_{X\succ 0} \; F(X)= f(X)-h(X).
\end{equation}
Introducing an epigraph variable yields:
\begin{equation}
\label{prob:lrmp-regu-mle-fw-objective}
\min_{X,y}\; y - h(X)
\quad
\text{s.t.}\quad
f(X)\le y .
\end{equation}

Applying the FW algorithm, construct the Lagrangian function and the linear minimization step at $X_t$ solves \cite{yurtsever2022cccp}:
\begin{equation}
\min_{X,y}
\;
y - \langle \nabla h(X_t), X \rangle
\quad
\text{s.t.}\quad
f(X)\le y,
\end{equation}
then we can get $\nabla f(X_t^*)=\nabla h(X_t)$. Since:
\begin{align}
\nabla f(X)
&=
-(1+\frac{\gamma}{n-1})X^{-1},
\\
\nabla h(X)
&=
-\frac{n-1}{K}\sum_{k=1}^K
\frac{y_k y_k^\top}{y_k^\top X y_k}
-
\frac{\gamma}{\mathrm{tr}(X)}I,
\end{align}
we obtain:
\begin{align}
(1+\frac{\gamma}{n-1})X_{t+1}^{-1}
=
\frac{n-1}{K}\sum_{k=1}^K
\frac{y_k y_k^\top}{y_k^\top X_t y_k}
+
\frac{\gamma}{\mathrm{tr}(X_t)}I.
\end{align}
That is:
\begin{equation}
\label{eq:fw-mle-sigma-iteration}
\Sigma_{t+1}
=
\frac{1}{1+\gamma/({n-1})}
\Big(
\frac{n-1}{K}\sum_{k=1}^K
\frac{y_k y_k^\top}{y_k^\top \Sigma_t^{-1} y_k}
+
\frac{\gamma}{\mathrm{tr}(\Sigma_t^{-1})}I
\Big),
\end{equation}
which coincides with the iteration update~\eqref{eq:cccp-regu-mle-sigma-iteration}. Thus, CCCP is equivalent to applying the FW algorithm for \eqref{prob:regu-mle-tyler-X-form}. In fact, FW-based methods have increasingly been used to solve related problems arising from CCCP frameworks.
\end{remark}

\begin{theorem} %[Existence of a fixed point via nonlinear Perron--Frobenius]
\label{the:npf-tyler}
Let $\{y_k\}_{k=1}^K \subset \mathbb R^{n-1}$ satisfy $\sum_{k=1}^K y_k y_k^\top \succ 0$. For $\gamma \ge 0$, define the mapping $M_\gamma:\mathbb S_{++}^{n-1} \to \mathbb S_{++}^{n-1}$ by:
\begin{align}
M_\gamma(\Sigma)
=
\frac{1}{1+\gamma/({n-1})}
\Big(
\frac{n-1}{K}\sum_{k=1}^K
\frac{y_k y_k^\top}{y_k^\top \Sigma^{-1} y_k}
+
\frac{\gamma}{\mathrm{tr}(\Sigma^{-1})}I
\Big).
\end{align}
Then $M_\gamma$ is continuous, order-preserving, positively homogeneous of degree one, and strongly positive. Consequently, there exist $\Sigma_\star \succ 0$ and $\alpha_\star>0$ such that $M_\gamma(\Sigma_\star)=\alpha_\star \Sigma_\star$. Moreover, the trace-normalized matrix $\widehat\Sigma_\star \triangleq \frac{n-1}{\mathrm{tr}(\Sigma_\star)}\Sigma_\star$ satisfies $\mathrm{tr}(\widehat\Sigma_\star)={n-1}$ and is a fixed point of the normalized mapping:
\begin{align}
\widehat M_\gamma(\Sigma)
\triangleq
\frac{(n-1)\,M_\gamma(\Sigma)}{\mathrm{tr}(M_\gamma(\Sigma))},
\end{align}
i.e., $\widehat M_\gamma(\widehat\Sigma_\star)=\widehat\Sigma_\star$.
\end{theorem}

\begin{proof}
Firstly, continuity follows from continuity of $\Sigma \mapsto \Sigma^{-1}$ on $\mathbb S_{++}^{n-1}$. Secondly, for $\alpha>0$, using $(\alpha\Sigma)^{-1}=\alpha^{-1}\Sigma^{-1}$ gives:
\begin{align}
M_\gamma(\alpha\Sigma)
&=
\frac{1}{1+\gamma/({n-1})}
(
\alpha
\frac{n-1}{K}\sum_{k=1}^K
\frac{y_k y_k^\top}{y_k^\top \Sigma^{-1} y_k}   \notag \\
&+
\alpha
\frac{\gamma}{\mathrm{tr}(\Sigma^{-1})}I
)
=
\alpha M_\gamma(\Sigma),
\end{align}
so $M_\gamma$ is positively homogeneous. Meanwhile, if $\Sigma_1 \preceq \Sigma_2$, then $\Sigma_1^{-1}\succeq \Sigma_2^{-1}$, which implies:
\begin{align}
\frac{y_k y_k^\top}{y_k^\top \Sigma_1^{-1} y_k}
\preceq
\frac{y_k y_k^\top}{y_k^\top \Sigma_2^{-1} y_k},
\end{align}
for all $k$. Moreover, $\mathrm{tr}(\Sigma_1^{-1}) \ge \mathrm{tr}(\Sigma_2^{-1})$, which yields:
\begin{align}
\frac{1}{\mathrm{tr}(\Sigma_1^{-1})}I
\preceq
\frac{1}{\mathrm{tr}(\Sigma_2^{-1})}I .
\end{align}

Hence $M_\gamma(\Sigma_1)\preceq M_\gamma(\Sigma_2)$ and $M_\gamma$ is order-preserving. Since $y_k^\top\Sigma^{-1}y_k>0$ for all $k$, the first term in $M_\gamma(\Sigma)$ is a positive weighted sum of $\{y_k y_k^\top\}_{k=1}^K$. Under the assumption $\sum_{k=1}^K y_k y_k^\top\succ0$, this term is positive definite. For $\gamma>0$, the identity term provides additional strict positivity. Thus $M_\gamma(\Sigma)\succ0$ for all $\Sigma\in\mathbb S_{++}^{n-1}$, and $M_\gamma$ is strongly positive.

%Additionally, since $y_k^\top \Sigma^{-1} y_k>0$ and $I\succ0$, the weighted sum defining $M_\gamma(\Sigma)$ is positive definite, implying strong positivity.

Finally, according to Nonlinear Perron--Frobenius (NPF) theory \cite{tan2015wireless, lemmens2012nonlinear}, there exist $\Sigma_\star\succ0$ and $\alpha_\star>0$ such that $M_\gamma(\Sigma_\star)=\alpha_\star\Sigma_\star$. Let $\widehat\Sigma_\star \triangleq \frac{n-1}{\mathrm{tr}(\Sigma_\star)}\Sigma_\star$ and define:
\begin{align}
\widehat M_\gamma(\Sigma)
=
\frac{(n-1)\,M_\gamma(\Sigma)}{\mathrm{tr}(M_\gamma(\Sigma))}.
\end{align}
By homogeneity, we can conclude $\widehat M_\gamma(\widehat\Sigma_\star)=\widehat\Sigma_\star$.
\end{proof}

%\vspace{1mm}

\subsubsection{Shrinkage-Stabilized CCCP Iteration}
Classical Tyler iterations are well defined when $K \ge n-1$, but may become ill-posed or unstable in undersampled regimes ($K<n-1$)~\cite{wiesel2011unified}. While the regularized objective~\eqref{prob:lrmp-regu-mle-Sigma-objective} improves spectral stability, additional shrinkage can further enhances robustness. Inspired by~\cite{chen2011robust}, we incorporate an isotropic shrinkage step. Given $\Sigma_t$, we first compute a stabilized update:
\begin{align}
\label{eq:mle-shrinkage-tyler}
F_\gamma(\Sigma_t)
=
&\frac{1}{1+\gamma/({n-1})}
\Bigl(
\frac{n-1}{K}\sum_{k=1}^K\frac{y_k y_k^\top}{y_k^\top \Sigma_t^{-1} y_k + \varepsilon} \notag \\
&+
\frac{\gamma}{\mathrm{tr}(\Sigma_t^{-1})}I
\Bigr),
\end{align}
where $\varepsilon>0$ is a small safeguard parameter for numerical stability. Shrinkage is then applied to the iterate:
\begin{align}
\label{eq:mle-shrinkage-sigma-iteration}
\widetilde{\Sigma}_{t+1}
=
(1-\rho)F_\gamma(\Sigma_t)
+
\rho I ,
\end{align}
where $\rho\in[0,1]$ is a shrinkage parameter used only for stabilization, while the primary spectral regularization is controlled by $\gamma$ in~\eqref{prob:lrmp-regu-mle-Sigma-objective}. In practice, $\rho$ is set to zero or a very small value when $K\ge {n-1}$, and activated when $K<{n-1}$.

Finally, trace normalization is applied:
\begin{align}
\label{eq:mle-shrinkage-sigma-normalization}
\Sigma_{t+1}
=
\frac{\widetilde{\Sigma}_{t+1}}
{\mathrm{tr}(\widetilde{\Sigma}_{t+1})/({n-1})}.
\end{align}
When $\rho=0$ and $\varepsilon=0$, the iteration reduces to the regularized Tyler update~\eqref{eq:cccp-regu-mle-sigma-iteration}, whose fixed point follows from NPF theory (Theorem~\ref{the:npf-tyler}). For $\rho\in(0,1]$, although shrinkage breaks positive homogeneity of the mapping, a fixed point still exists as shown in Theorem~\ref{the:fixed-point-shrinkage-tyler}.

\begin{figure}[!t]
    \centering
    \setlength{\tabcolsep}{2pt}
    \renewcommand{\arraystretch}{1.0}
    \subfloat[\footnotesize Eigenvalues]{%
        \includegraphics[width=0.24\textwidth]{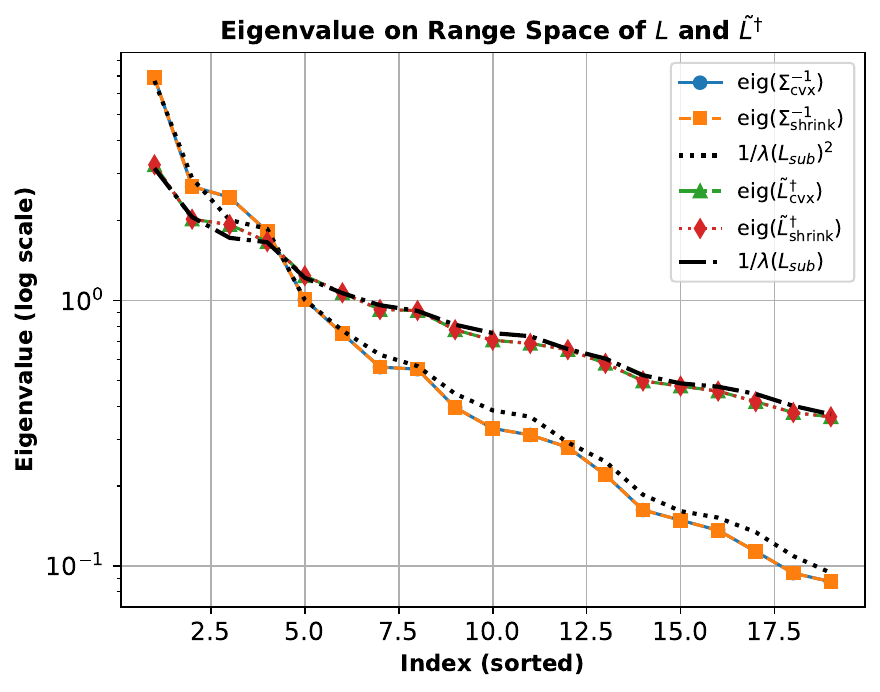}%
        \label{fig:phase1_eigenvalues}
    }%
    \subfloat[\footnotesize Spectral Alignment]{%
        \includegraphics[width=0.23\textwidth]{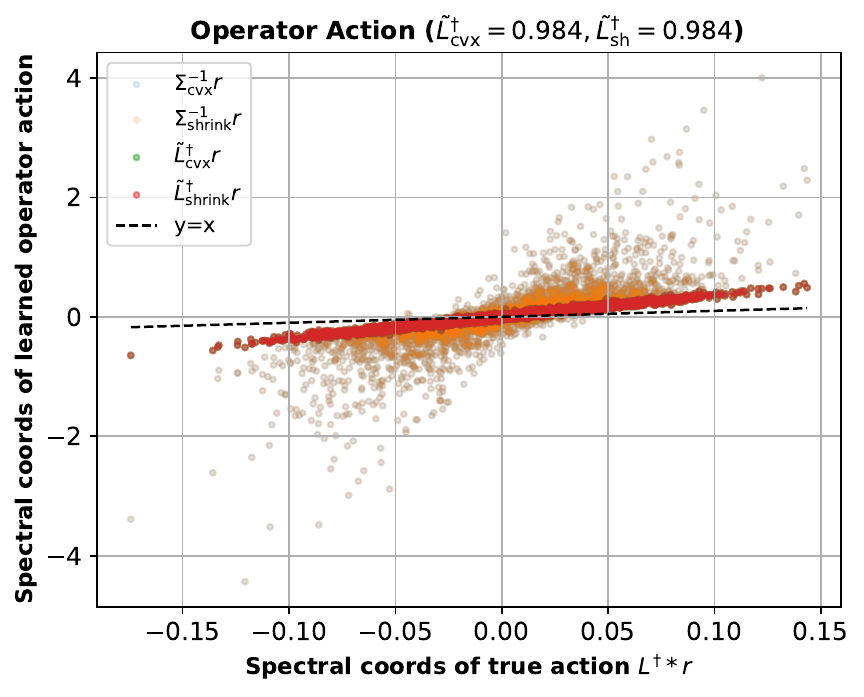}%
        \label{fig:phase1_spectral}
    }
    \caption{MLE on $\mathcal R(L)$. (a) $\Sigma^{-1}$ follows the theoretical eigenvalue scaling of $(L^2)^{\dagger}$ on $\mathcal R(L)$. (b): after $\tilde{L}^{\dagger}=\sqrt{\Sigma^{-1}}$, $\tilde{L}^{\dagger}$ aligns with $L^\dagger$ in spectral coordinates, confirming spectrum approximation of $L^{\dagger}$.}
    \label{fig:laplacian-mle-tyler-and-shrinkage-tyler}
\end{figure}

\begin{theorem}%[Existence of a fixed point for shrinkage-stabilized Tyler iteration]
\label{the:fixed-point-shrinkage-tyler}
Let $\varepsilon>0$, $\rho\in(0,1]$, and $\gamma\ge0$. Assume $\sum_{k=1}^K y_k y_k^\top \succ 0$ and define:
\begin{align}
F_\gamma(\Sigma)
&=
\frac{1}{1+\gamma/(n-1)}
\Big(
\frac{n-1}{K}\sum_{k=1}^K
\frac{y_k y_k^\top}{y_k^\top \Sigma^{-1}y_k+\varepsilon}
+
\frac{\gamma}{\mathrm{tr}(\Sigma^{-1})}I
\Big),\\
\widetilde M(\Sigma)
&=
(1-\rho)F_\gamma(\Sigma)+\rho I,\\
M(\Sigma)
&=
\frac{(n-1)\widetilde M(\Sigma)}
{\mathrm{tr}(\widetilde M(\Sigma))}.
\end{align}
Then $M$ admits at least one fixed point $\Sigma_\star\succ0$ satisfying $\mathrm{tr}(\Sigma_\star)=n-1$ and $M(\Sigma_\star)=\Sigma_\star$.
\end{theorem}

\begin{proof}
Let $d=n-1$. For any $\Sigma\succ0$ with $\mathrm{tr}(\Sigma)=d$, we have $y_k^\top\Sigma^{-1}y_k+\varepsilon\ge\varepsilon$. Moreover, by $\mathrm{tr}(\Sigma)\mathrm{tr}(\Sigma^{-1})\ge d^2$, we have $\mathrm{tr}(\Sigma^{-1})\ge d$. Hence,
\begin{align}
\mathrm{tr}(F_\gamma(\Sigma))
&\le
\frac{1}{1+\gamma/d}
\Big(
\frac{d}{K\varepsilon}
\sum_{k=1}^K \|y_k\|_2^2
+\gamma
\Big)
\triangleq C_F .
\end{align}
Therefore, $\mathrm{tr}(\widetilde M(\Sigma)) \le (1-\rho)C_F+\rho d \triangleq C$, while $\widetilde M(\Sigma) = (1-\rho)F_\gamma(\Sigma)+\rho I \succeq \rho I$. Consequently,
\begin{align}
M(\Sigma)
=
\frac{d\,\widetilde M(\Sigma)}
{\mathrm{tr}(\widetilde M(\Sigma))}
\succeq
\frac{d\rho}{C}I.
\end{align}

Let $\delta=d\rho/C>0$ and define $\mathcal S_\delta = \left\{\Sigma\succeq\delta I:\mathrm{tr}(\Sigma)=d \right\}$. The set $\mathcal S_\delta$ is nonempty, convex, and compact. Furthermore, the above bounds imply that $M(\mathcal S_\delta)\subseteq\mathcal S_\delta$. Since every matrix in $\mathcal S_\delta$ is positive definite, all operations defining $F_\gamma$, $\widetilde M$, and $M$ are continuous on $\mathcal S_\delta$. Therefore, by Brouwer's fixed-point theorem \cite{granas2003fixed}, there exists $\Sigma_\star\in\mathcal S_\delta$ such that $M(\Sigma_\star)=\Sigma_\star$. Since $\mathcal S_\delta\subset\mathbb S_{++}^{d}$ and $\mathrm{tr}(\Sigma_\star)=d=n-1$, the result follows.
\end{proof}

\begin{example}
We evaluate the shrinkage-regularized CCCP iteration \eqref{eq:mle-shrinkage-tyler}--\eqref{eq:mle-shrinkage-sigma-normalization} for \eqref{prob:lrmp-regu-mle-Sigma-objective} with $\mathrm{tr}(\Sigma)=n-1$ on a connected random graph ($n=20$). We generate $K=10n=200$ samples $r_k=L z_k$, normalize $\bar r_k=r_k/\|r_k\|_2$, and obtain intrinsic coordinates $\{y_k\}_{k=1}^K$ on $\mathcal R(L)$. A CVXPY-based MM solver is used as a benchmark. Since $K\ge n-1$, the method reduces to a stabilized Tyler iteration ($\rho=10^{-6}$) and converges to the same stationary point. The resulting $\Sigma^{-1}$ closely matches the benchmark (normalized Frobenius inner product $1.0000$, relative gap $9.7\times10^{-5}$). As shown in Fig.~\ref{fig:laplacian-mle-tyler-and-shrinkage-tyler}, $\Sigma^{-1}$ aligns with $(L^2)^{-1}$ on $\mathcal R(L)$ (similarity $\approx0.987$), and $\tilde{L}^{\dagger}=\sqrt{\Sigma^{-1}}$ matches $L^{\dagger}$ (similarity $\approx0.992$), consistent with theory (cf. Fig.~\ref{fig:phase1_eigenvalues}). Spectral tests (cf. Fig.~\ref{fig:phase1_spectral}) further show strong agreement with $L^{\dagger}r$ ($\approx0.984$), confirming accurate recovery on $\mathcal R(L)$.
\end{example}

\subsection{Differentiable Primal Reconstruction}
\label{subsec:phase2}
In the second phase, the learned $\tilde{L}^{\dagger}$ is embedded into a differentiable primal optimization pipeline, where it acts as a graph-aware preconditioner to solve the original LR-NNLS problem~\eqref{prob:Learning_LRMP_primal} directly in the primal space. Let:
\begin{equation}
\label{eq:phase2-gradient}
g(x)
=
A^\top(Ax-b)+Lx
\end{equation}
denote the gradient of the smooth objective. Since $\tilde{L}^{\dagger}$ approximates the action of the Laplacian pseudoinverse on $\mathcal{R}(L)=\mathbf{1}^{\perp}$, we first project the gradient onto this subspace:
\begin{equation}
\label{eq:projected-gradient-range}
g_{\perp}
=
P_{\perp}g(x),
\qquad
P_{\perp}
=
I-\frac{1}{n}\mathbf{1}\mathbf{1}^{\top}.
\end{equation}

The learned graph-aware direction is given by $p=P_{\perp}\widetilde{L}^{\dagger}g_{\perp}$. We normalize it to the scale of the objective gradient and form:
\begin{equation}
\label{eq:phase2-blended-direction}
\widehat{p}
=
\frac{\|g(x)\|_2}
{\|p\|_2+\epsilon}\,p,
\qquad
d
=
(1-\eta)g(x)+\eta\widehat{p},
\end{equation}
where $\eta\in[0,1)$ controls the preconditioning contribution. The primal iterate is then updated as:
\begin{equation}
\label{eq:phase2-primal-update}
x^{+}
=
\left[x-\alpha d\right]_{+},
\end{equation}
where $\alpha$ is initialized from the local quadratic model and refined by projected backtracking. If the preconditioned step is not accepted, the method falls back to the projected-gradient direction $d=g(x)$ for a descent-safe update. Convergence is monitored using the projected-gradient mapping:
\begin{equation}
\label{eq:phase2-pg-mapping}
G(x)
=
x-\left[x-g(x)\right]_{+},
\end{equation}
which vanishes at a first-order solution of the problem.

\section{Numerical Experiments}

\subsection{Experimental Setup}
\label{subsec:sim-setup}
We evaluate DCR on the LR-NNLS problem~\eqref{prob:Learning_LRMP_primal} with $n\in\{50,100,200,500,700,1000,1500,2000\}, \; m=1.5n$. The measurement matrix $A\in\mathbb{R}^{m\times n}$ contains i.i.d.\ standard Gaussian entries and is normalized column-wise. The ground-truth signal $x_{\mathrm{gt}}\in\mathbb{R}^n$ is $10\%$ sparse with positive nonzero entries, and the measurements are generated as $b=Ax_{\mathrm{gt}}$.
We consider two graph structures: a perturbed two-dimensional grid (grid2d) and an Erd\H{o}s--R\'enyi (ER) graph. For each graph, Phase~1 learns the graph Laplacian pseudoinverse $\widetilde{L}^{\dagger}$ once using \eqref{eq:mle-shrinkage-tyler}--\eqref{eq:mle-shrinkage-sigma-normalization}. The number of probing samples is selected adaptively as:
\begin{equation}
\label{eq:adaptive-K-exp}
K(n)
=
\min\left\{
K_{\max},
\max\left\{
K_{\min},
\operatorname{round}(2\sqrt{n})
\right\}
\right\}.
\end{equation}
Phase~2 then solves the LR-NNLS problem directly in the primal space using the learned-preconditioned projected update described in Section~\ref{subsec:phase2}. High-accuracy reference solutions, used only for evaluation, are obtained using CVXPY~\cite{diamond2016cvxpy} with OSQP, with SCS used as a fallback solver. The main experimental parameters are summarized in Table~\ref{tab:DCR-topo-params}.

\begin{table}[!t]
\caption{Experimental Parameters.}
\label{tab:DCR-topo-params}
\centering
\renewcommand{\arraystretch}{0.9}
\begin{tabularx}{\linewidth}{@{}X l@{}}
\toprule
\textbf{Parameter} & \textbf{Value} \\
\midrule
Problem size $n$
& $\{50,100,200,500,700,1000,1500,2000\}$ \\
Measurements $m$
& $1.5n$ \\
Graph topology
& grid2d, ER \\
Grid weight jitter
& $0.2$ \\
ER edge probability
& $\mathrm{clip}(0.03\times200/n,\,0.005,\,0.5)$ \\
Phase~1 iterations
& $\min\{3000,\,500+2(n-1)\}$ \\
Probing samples $K$
& $\mathrm{clip}(\mathrm{round}(2\sqrt{n}),16,96)$ \\
Regularization $\gamma$
& $10^{-2}$ \\
Phase~1 shrinkage $\rho$
& Adaptive, $\rho\in[0,0.02]$ \\
Stabilization $(\epsilon,\delta)$
& $(10^{-6},10^{-3})$ \\
Phase~2 maximum iterations
& $5000$ \\
Preconditioner blending $\eta$
& $0.5$ \\
Projected-gradient tolerance
& $10^{-6}$ \\
Armijo parameter
& $10^{-4}$ \\
Backtracking factor
& $0.5$ \\
Chebyshev order
& $30$ \\
CVXPY tolerance
& $10^{-5}$ \\
\bottomrule
\end{tabularx}
\end{table}

\subsubsection{Baselines}
We compare DCR with:
\begin{itemize}
\item CVXPY~\cite{diamond2016cvxpy}: a high-accuracy convex solver used as the reference solution.

\item PGD \cite{boyd2004convex, tao2026learning}: projected gradient descent with the same backtracking step-size rule as DCR.

\item FISTA \cite{beck2009fast}: an accelerated proximal-gradient method for the nonnegative constrained quadratic objective.

\item ADMM \cite{boyd2011distributed}: an operator-splitting method based on the variable splitting $x=z$, where the quadratic $x$-subproblem is solved using a cached factorization.

\item Chebyshev \cite{hammond2011wavelets, shuman2018distributed, cheng2025iterative}: the action of $L^{\dagger}$ is approximated using a degree-$30$ Chebyshev polynomial and incorporated into a projected iterative solver.
\end{itemize}

We additionally compare $\widetilde{L}^{\dagger}$ against an eigendecomposition-based exact $L^{\dagger}$ implementation in a separate operator ablation.

\subsubsection{Metrics}
Solution accuracy is measured relative to the CVXPY reference:
\begin{align}
\mathrm{RelErr}
&=
\frac{
\|x-x_{\mathrm{CVXPY}}\|_2
}{
\|x_{\mathrm{CVXPY}}\|_2
},
\label{eq:simu-rel-sol-err}
\\
\mathrm{RelGap}
&=
\frac{
|O(x)-O(x_{\mathrm{CVXPY}})|
}{
|O(x_{\mathrm{CVXPY}})|
},
\label{eq:simu-rel-obj-gap}
\end{align}
where $O(\cdot)$ denotes the objective of \eqref{prob:Learning_LRMP_primal}.
To compare online convergence, we report the time required to first reach $\mathrm{RelGap}\le 10^{-3}$ and $\mathrm{RelGap}\le 10^{-6}$, denoted by $T_{10^{-3}}$ and $T_{10^{-6}}$, respectively. For DCR, these quantities measure Phase~2 online time only. The one-time Phase~1 cost is reported separately and is not hidden in the end-to-end analysis.

\subsection{Experimental Results}

\subsubsection{Overall Solution Accuracy}
\label{subsec:results-accuracy}

\begin{figure}[!t]
    \centering
    \includegraphics[trim=0.2cm 0.2cm 0.2cm 1.0cm, clip, width=0.49\textwidth]{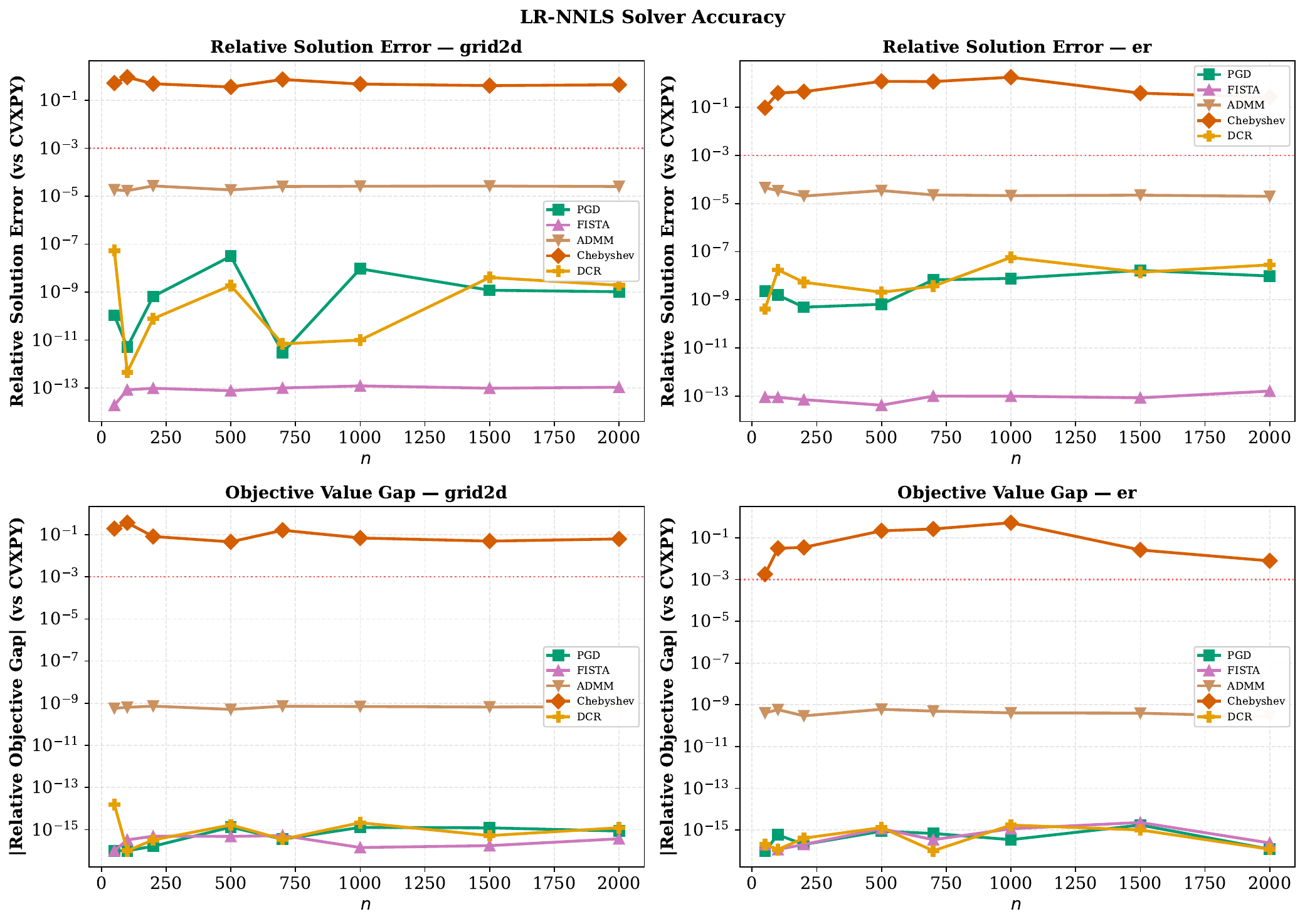}
    \caption{Solution accuracy of PGD, FISTA, ADMM, Chebyshev, and DCR for $n\in[50,2000]$ on grid2d and ER graphs. The top and bottom rows report the relative solution error in~\eqref{eq:simu-rel-sol-err} and absolute relative objective gap in~\eqref{eq:simu-rel-obj-gap}, respectively, with respect to the CVXPY reference. The dashed line at $10^{-3}$ indicates a moderate-accuracy threshold.}
    \label{fig:accuracy}
\end{figure}

Fig.~\ref{fig:accuracy} compares the final solution accuracy across all problem sizes. DCR consistently recovers solutions close to the CVXPY reference under both graph topologies. On grid2d graphs, its relative solution error ranges from $4.61\times10^{-13}$ to $5.45\times10^{-8}$ for $n=50$--$2000$, with objective gaps generally near numerical precision. At $n=500$ and $n=2000$, DCR achieves solution errors of $1.90\times10^{-9}$ and $1.94\times10^{-9}$, with objective gaps of $1.68\times10^{-15}$ and $1.23\times10^{-15}$, respectively.
Similar accuracy is observed on ER graphs, where DCR maintains relative solution errors between $4.14\times10^{-10}$ and $5.74\times10^{-8}$, again with objective gaps near numerical precision. At $n=500$ and $n=2000$, the solution errors are $2.05\times10^{-9}$ and $2.82\times10^{-8}$, with objective gaps of $1.29\times10^{-15}$ and $1.18\times10^{-16}$, respectively. These results show that the learned Laplacian pseudoinverse of DCR preserves high-quality primal reconstruction across different graph structures.
Meanwhile, baselines exhibit distinct accuracy characteristics. FISTA generally achieves the smallest solution errors, typically around $10^{-13}$, while PGD also reaches high accuracy, with errors approximately between $10^{-12}$ and $10^{-8}$. ADMM attains small objective gaps but terminates with solution errors around $10^{-5}$ under its stopping criterion. In contrast, Chebyshev produces substantially larger errors, with solution errors typically between $10^{-1}$ and $1$. Overall, DCR achieves accuracy comparable to the high-accuracy first-order methods without requiring an exact Laplacian pseudoinverse.

{\setlength{\heavyrulewidth}{0.12em}
\begin{table*}[!t]
\caption{Representative accuracy and online time-to-accuracy results for the LR-NNLS problem. RelErr and RelGap are measured relative to the CVXPY reference. $T_{10^{-3}}$ and $T_{10^{-6}}$ denote the online time required to first reach the corresponding relative objective-gap threshold. ``--'' indicates that the target accuracy is not reached.}
\label{tab:main-performance}
\centering
\renewcommand{\arraystretch}{0.9}
\setlength{\tabcolsep}{3.2pt}

\resizebox{\textwidth}{!}{%
\begin{tabular}{c l c c c c @{\hspace{8mm}} c l c c c c}
\toprule
\multicolumn{6}{c}{\textbf{grid2d}}
&
\multicolumn{6}{c}{\textbf{ER}}
\\
\cmidrule(lr){1-6}
\cmidrule(lr){7-12}

$\boldsymbol{n}$
& \textbf{Method}
& \textbf{RelErr}
& \textbf{RelGap}
& $\boldsymbol{T_{10^{-3}}}$
& $\boldsymbol{T_{10^{-6}}}$

&
$\boldsymbol{n}$
& \textbf{Method}
& \textbf{RelErr}
& \textbf{RelGap}
& $\boldsymbol{T_{10^{-3}}}$
& $\boldsymbol{T_{10^{-6}}}$
\\
\midrule

\multirow{5}{*}{500}
& PGD
& $3.16\times10^{-8}$
& $1.32\times10^{-15}$
& 0.0596
& 0.1126
&
\multirow{5}{*}{500}
& PGD
& $6.48\times10^{-10}$
& $8.60\times10^{-16}$
& 0.0490
& 0.1190
\\

& FISTA
& $7.80\times10^{-14}$
& $4.80\times10^{-16}$
& 0.0816
& 0.1052
&
& FISTA
& $4.13\times10^{-14}$
& $1.15\times10^{-15}$
& 0.0170
& 0.0230
\\

& ADMM
& $1.84\times10^{-5}$
& $5.10\times10^{-10}$
& 0.0241
& 0.0458
&
& ADMM
& $3.47\times10^{-5}$
& $6.06\times10^{-10}$
& 0.0095
& 0.0220
\\

& Chebyshev
& $3.61\times10^{-1}$
& $4.54\times10^{-2}$
& --
& --
&
& Chebyshev
& $1.20$
& $2.18\times10^{-1}$
& --
& --
\\

& DCR
& $1.90\times10^{-9}$
& $1.68\times10^{-15}$
& 0.0499
& 0.3011
&
& DCR
& $2.05\times10^{-9}$
& $1.29\times10^{-15}$
& 0.0381
& 0.1154
\\

\cmidrule(lr){1-6}
\cmidrule(lr){7-12}

\multirow{5}{*}{2000}
& PGD
& $1.04\times10^{-9}$
& $8.64\times10^{-16}$
& 0.2940
& 0.4036
&
\multirow{5}{*}{2000}
& PGD
& $9.73\times10^{-9}$
& $1.18\times10^{-16}$
& 0.3957
& 0.4719
\\

& FISTA
& $1.08\times10^{-13}$
& $3.70\times10^{-16}$
& 0.0669
& 0.0824
&
& FISTA
& $1.60\times10^{-13}$
& $2.37\times10^{-16}$
& 0.0741
& 0.0901
\\

& ADMM
& $2.55\times10^{-5}$
& $6.95\times10^{-10}$
& 0.1573
& 0.1793
&
& ADMM
& $2.00\times10^{-5}$
& $2.73\times10^{-10}$
& 0.1547
& 0.1677
\\

& Chebyshev
& $4.47\times10^{-1}$
& $6.24\times10^{-2}$
& --
& --
&
& Chebyshev
& $2.62\times10^{-1}$
& $7.92\times10^{-3}$
& 0.9097
& --
\\

& DCR
& $1.94\times10^{-9}$
& $1.23\times10^{-15}$
& 0.0360
& 0.3392
&
& DCR
& $2.82\times10^{-8}$
& $1.18\times10^{-16}$
& 0.0871
& 0.2378
\\

\bottomrule
\end{tabular}%
}
\end{table*}
}

Table~\ref{tab:main-performance} further summarizes representative accuracy and time-to-accuracy results for $n=500$ and $n=2000$. At the moderate $10^{-3}$ objective-gap threshold, DCR shows favorable performance on large instances. For grid2d with $n=2000$, DCR reaches the threshold in $0.0360$\,s, compared with $0.2940$\,s for PGD, $0.0669$\,s for FISTA, and $0.1573$\,s for ADMM. On ER, DCR requires $0.0871$\,s, compared with $0.3957$\,s for PGD and $0.1547$\,s for ADMM, while FISTA is slightly faster at $0.0741$\,s. At the stricter $10^{-6}$ threshold, FISTA and ADMM generally reach high precision faster than DCR, while DCR remains competitive with PGD on large instances. At $n=2000$, DCR requires $0.3392$\,s on grid2d and $0.2378$\,s on ER, compared with $0.4036$\,s and $0.4719$\,s for PGD, respectively. FISTA reaches the same threshold in $0.0824$\,s and $0.0901$\,s, and ADMM in $0.1793$\,s and $0.1677$\,s. Thus, DCR's online advantage is most pronounced in the early-to-moderate accuracy regime rather than in high-precision refinement.

\subsubsection{Online Convergence and Time-to-Accuracy}
\label{subsec:results-convergence}

\begin{figure}[!t]
    \centering
    \subfloat[\footnotesize grid2d topology]{\includegraphics[trim=0.2cm 0.2cm 0.2cm 1.0cm, clip, width=0.49\textwidth]{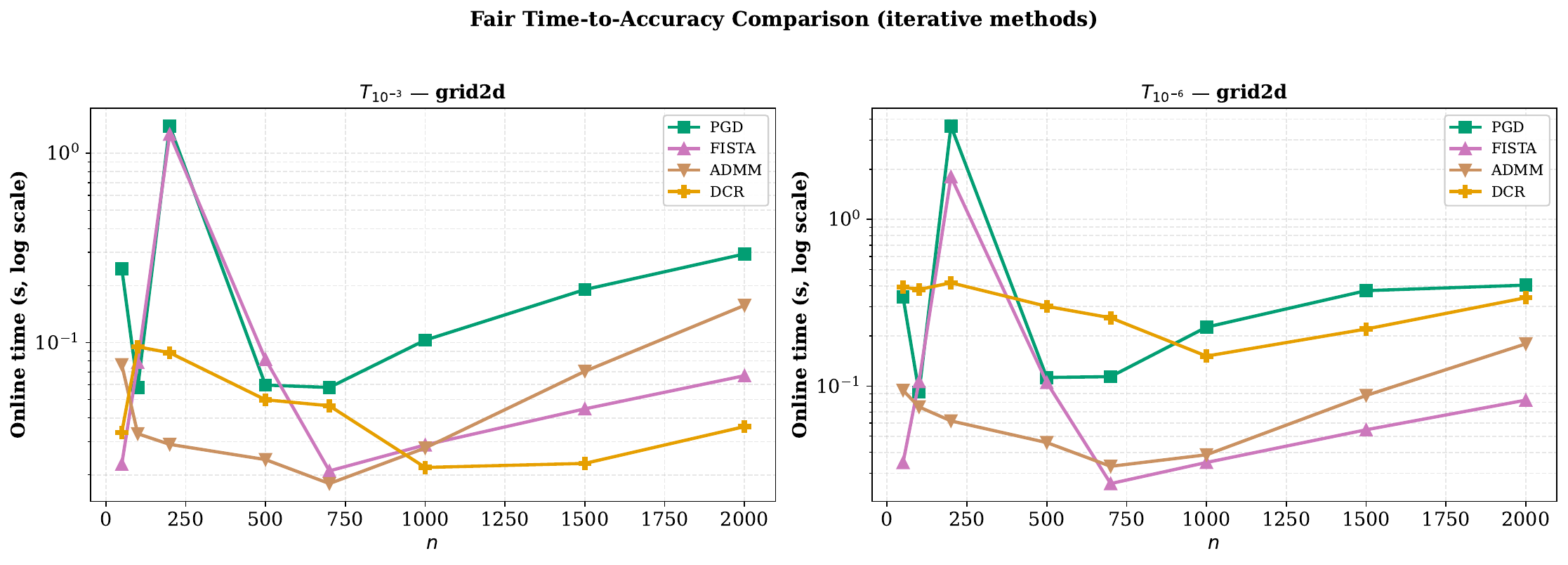}}
    \hfill
    \subfloat[\footnotesize ER topology]{\includegraphics[trim=0.2cm 0.2cm 0.2cm 1.0cm, clip, width=0.49\textwidth]{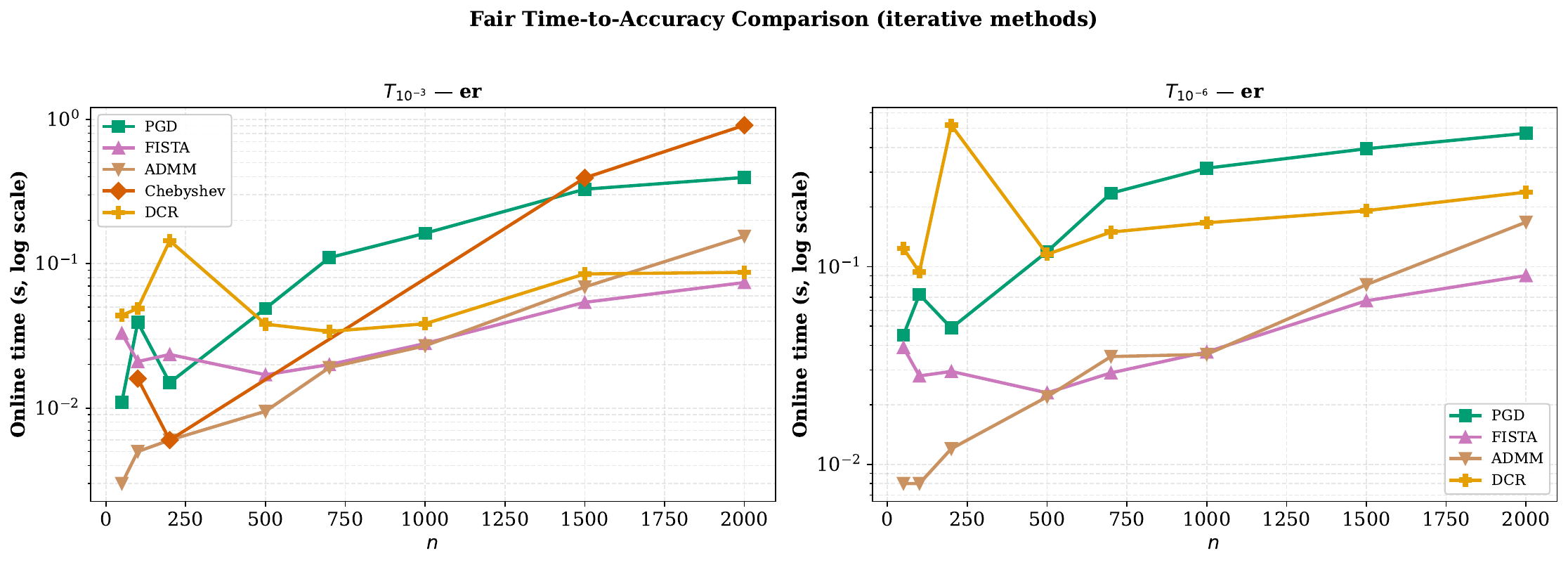}}
    \caption{Online time-to-accuracy comparison of the iterative methods. The left and right panels report $T_{10^{-3}}$ and $T_{10^{-6}}$, respectively. For DCR, only online primal-reconstruction time is included; pseudoinverse learning is treated as offline preprocessing.}
    \label{fig:time_accuracy}
\end{figure}

Fig.~\ref{fig:time_accuracy} compares the online time required to reach relative objective-gap thresholds of $10^{-3}$ and $10^{-6}$. At the $10^{-3}$ threshold, DCR shows favorable scaling, particularly on large grid2d problems. For $n=500$, DCR requires $0.0499$\,s, compared with $0.0596$\,s for PGD, $0.0816$\,s for FISTA, and $0.0241$\,s for ADMM. As $n$ increases, its runtime remains relatively stable at $0.0464$\,s, $0.0219$\,s, $0.0230$\,s, and $0.0360$\,s for $n=700$, $1000$, $1500$, and $2000$, respectively. At $n=2000$, DCR reaches the threshold earlier than PGD ($0.2940$\,s), FISTA ($0.0669$\,s), and ADMM ($0.1573$\,s).
The ER results show a similar but less uniform trend. At $n=500$, DCR reaches the $10^{-3}$ threshold in $0.0381$\,s, compared with $0.0490$\,s for PGD, $0.0170$\,s for FISTA, and $0.0095$\,s for ADMM. At $n=1500$, DCR requires $0.0850$\,s, versus $0.3281$\,s for PGD, $0.0540$\,s for FISTA, and $0.0690$\,s for ADMM. At $n=2000$, the corresponding times are $0.0871$\,s for DCR, $0.3957$\,s for PGD, $0.0741$\,s for FISTA, and $0.1547$\,s for ADMM. Thus, DCR remains competitive at moderate accuracy and becomes particularly favorable relative to PGD and ADMM as the problem size increases.

\begin{figure*}[!t]
    \centering
    \subfloat[\footnotesize grid2d topology]{
        \includegraphics[trim=0.2cm 0.2cm 0.2cm 1.0cm, clip, width=0.49\textwidth]{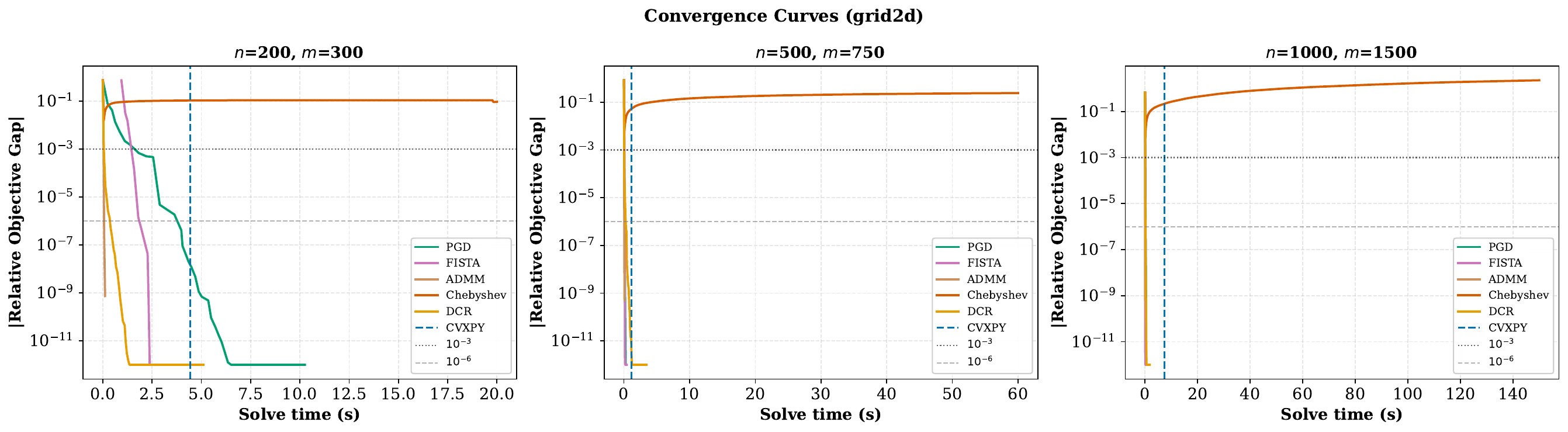}
    }%
    \subfloat[\footnotesize ER topology]{
        \includegraphics[trim=0.2cm 0.2cm 0.2cm 1.0cm, clip, width=0.49\textwidth]{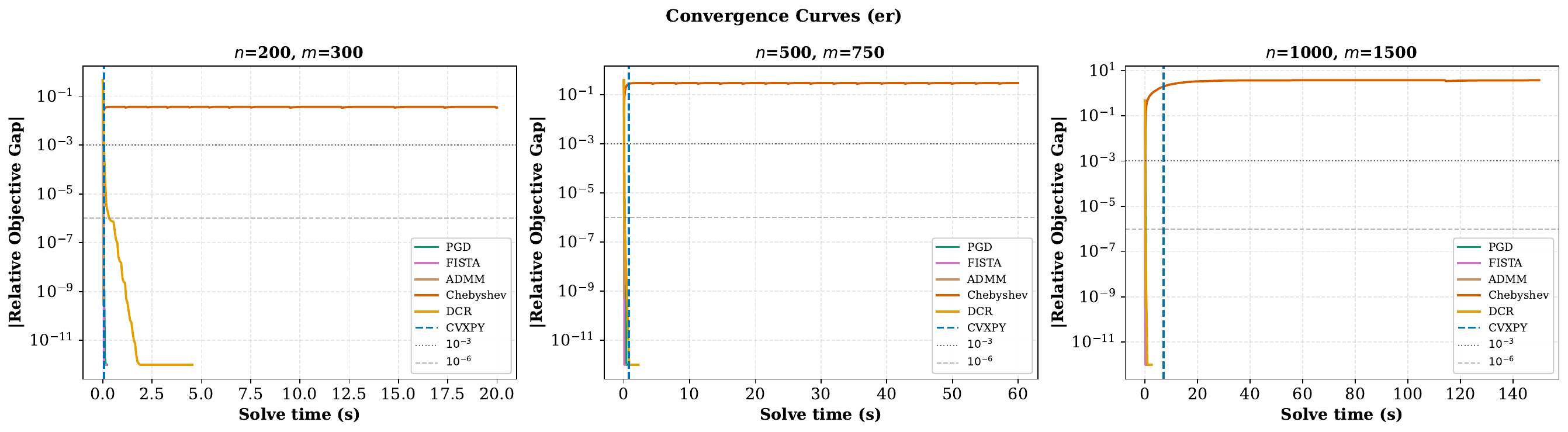}
    }
    \caption{Relative objective-gap trajectories for representative problem sizes $n\in\{200,500,1000\}$. The horizontal reference lines indicate objective-gap thresholds of $10^{-3}$ and $10^{-6}$, while the vertical dashed line indicates the CVXPY solve time for the corresponding instance.}
    \label{fig:conv_curves}
\end{figure*}

At the stricter $10^{-6}$ threshold, FISTA and ADMM generally reach high accuracy faster than DCR. For grid2d at $n=2000$, DCR requires $0.3392$\,s, compared with $0.4036$\,s for PGD, $0.0824$\,s for FISTA, and $0.1793$\,s for ADMM. On ER, the corresponding times are $0.2378$\,s, $0.4719$\,s, $0.0901$\,s, and $0.1677$\,s, respectively. DCR therefore remains faster than PGD in these large-scale cases, while its online advantage is more pronounced at moderate accuracy than during high-precision refinement.
Fig.~\ref{fig:conv_curves} further illustrates these convergence characteristics. Across both topologies, DCR rapidly reduces the objective gap in the early stage, crosses the $10^{-3}$ threshold, and continues toward the $10^{-6}$ regime and the near-reference solutions in Fig.~\ref{fig:accuracy}. PGD and FISTA eventually attain high accuracy, whereas ADMM terminates with a larger solution error despite its small objective gap. The fixed-order Chebyshev method frequently stagnates above the high-accuracy thresholds, consistent with its larger final errors in Fig.~\ref{fig:accuracy}. Overall, the learned pseudoinverse mainly improves the early-to-intermediate convergence of the primal solver rather than uniformly accelerating all accuracy regimes.

\subsubsection{Cross-Instance Reuse of the Learned Pseudoinverse}
\label{subsec:results-reuse}

\begin{figure}[!t]
    \centering
    \includegraphics[trim=0.2cm 0.2cm 0.2cm 1.0cm, clip, width=0.49\textwidth]{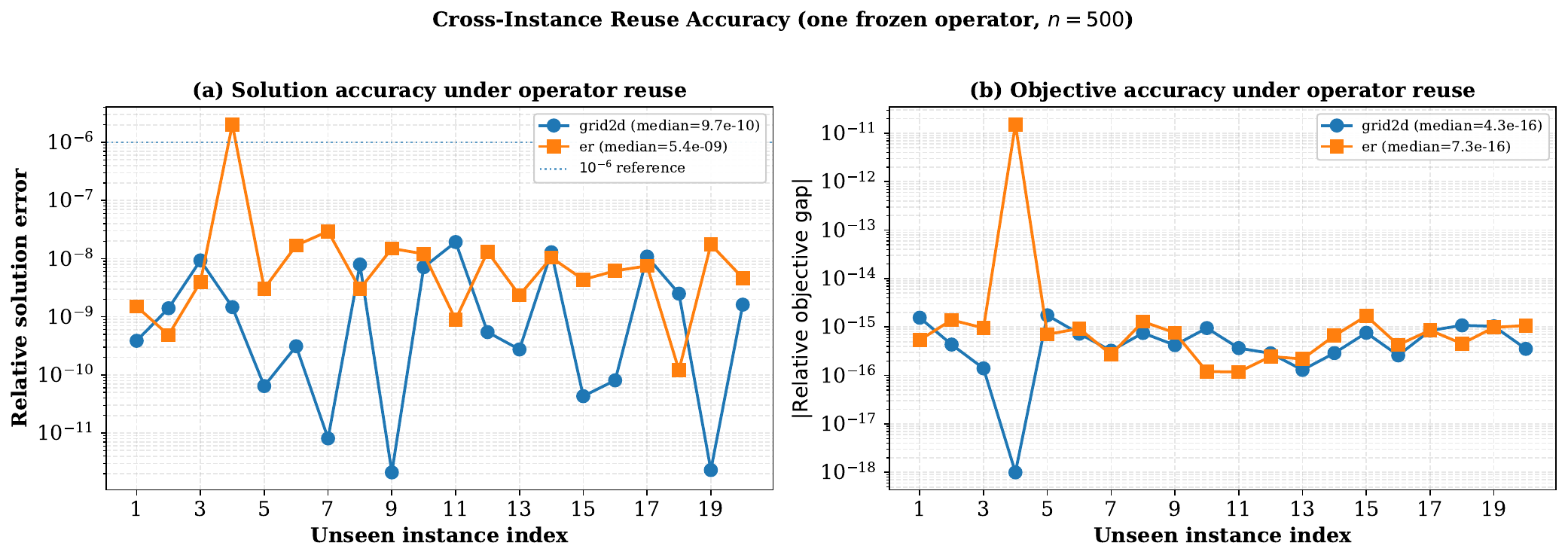}
    \caption{Cross-instance reuse accuracy of a frozen learned graph Laplacian pseudoinverse at $n=500$. For each topology, one $\widetilde{L}^{\dagger}$ is learned and reused unchanged across $20$ independently generated LR-NNLS instances with different $(A,b)$. Left: relative solution error. Right: absolute relative objective gap. CVXPY is used only as a post-hoc accuracy reference.}
    \label{fig:reuse_accuracy}
\end{figure}

A key property of DCR is that the learned $\widetilde{L}^{\dagger}$ depends only on the graph Laplacian $L$, rather than on a particular $(A,b)$ or reference solution. This allows the pseudoinverse to be learned once for a given graph and reused across downstream LR-NNLS instances sharing the same graph. To evaluate this property, we freeze one learned $\widetilde{L}^{\dagger}$ for each topology and apply it unchanged to $20$ independently generated instances with different $(A,b)$.
As shown in Fig.~\ref{fig:reuse_accuracy}, the frozen $\widetilde{L}^{\dagger}$ preserves high reconstruction accuracy across the unseen instances. For grid2d, the median relative solution error is $9.73\times10^{-10}$, while the median absolute relative objective gap is $4.29\times10^{-16}$. For ER, the corresponding values are $5.41\times10^{-9}$ and $7.33\times10^{-16}$, respectively. Most instances therefore remain extremely close to the CVXPY reference in both solution and objective value.
The ER topology exhibits slightly larger variability than grid2d. In particular, one ER instance attains a relative solution error of $2.04\times10^{-6}$, which raises the mean error to $1.10\times10^{-7}$. Nevertheless, the median error remains at $5.41\times10^{-9}$, and the objective gap of this harder instance is still only on the order of $10^{-11}$. This indicates that the larger solution deviation is isolated rather than representative of the typical reuse behavior.

\begin{table}[!t]
\caption{Cross-instance reuse accuracy at $n=500$ when a single learned graph operator is reused across $20$ independently generated LR-NNLS instances for each topology.}
\label{tab:reuse-summary}
\centering
\renewcommand{\arraystretch}{0.9}
\setlength{\tabcolsep}{3.5pt}

\resizebox{\linewidth}{!}{%
\begin{tabular}{l c c c c}
\toprule
\textbf{Topology}
& \textbf{Median RelErr}
& \textbf{Mean RelErr}
& \textbf{Worst RelErr}
% & \textbf{Median $|\mathrm{RelGap}|$}
& \textbf{Median RelGap}
\\
\midrule
grid2d
& $9.73\times10^{-10}$
& $3.80\times10^{-9}$
& $1.93\times10^{-8}$
& $4.29\times10^{-16}$ \\
ER
& $5.41\times10^{-9}$
& $1.10\times10^{-7}$
& $2.04\times10^{-6}$
& $7.33\times10^{-16}$ \\
\bottomrule
\end{tabular}%
}
\end{table}

Table~\ref{tab:reuse-summary} summarizes the cross-instance accuracy statistics. On grid2d, the mean and worst relative solution errors are $3.80\times10^{-9}$ and $1.93\times10^{-8}$, respectively, showing that the reuse accuracy remains tightly concentrated across the $20$ instances. On ER, the mean error increases to $1.10\times10^{-7}$ because of the single harder instance discussed above, while the median remains two orders of magnitude smaller at $5.41\times10^{-9}$. Despite this variability in solution error, the median absolute objective gaps remain close to numerical precision for both topologies. These results support the cross-instance reuse mechanism of DCR: the learned $\widetilde{L}^{\dagger}$ depends on the shared graph Laplacian $L$ rather than a specific downstream instance.

\subsubsection{Pseudoinverse Ablation and Offline Cost}
\label{subsec:results-ablation}
To isolate the effect of the learned $\widetilde{L}^{\dagger}$, we conduct an ablation under the same projected optimization procedure. We replace $\widetilde{L}^{\dagger}$ with (i) no graph preconditioner (PGD), (ii) a fixed-order Chebyshev approximation of $L^{\dagger}$, and (iii) the exact $L^{\dagger}$ obtained by spectral decomposition. All variants use the same initialization, projected update, line search, and stopping criterion.

{\setlength{\heavyrulewidth}{0.12em}
\begin{table}[!t]
\footnotesize
\caption{Pseudoinverse ablation under the same projected optimization procedure. RelErr is measured relative to CVXPY, while $T_{10^{-3}}$ and $T_{10^{-6}}$ denote the online time required to reach the corresponding relative objective-gap thresholds.}
\label{tab:pseudoinverse-ablation}
\centering
\renewcommand{\arraystretch}{0.9}
\setlength{\tabcolsep}{3.2pt}

\resizebox{\columnwidth}{!}{%
\begin{tabular}{c c l c c c}
\toprule
\textbf{Topology}
& $\boldsymbol{n}$
& \textbf{Method}
& \textbf{RelErr}
& $\boldsymbol{T_{10^{-3}}}$
& $\boldsymbol{T_{10^{-6}}}$ \\
\midrule

\multirow{8}{*}{grid2d}
& \multirow{4}{*}{200}
& PGD
& $5.13\times10^{-6}$ & 0.0020 & 0.0050 \\
&& Chebyshev
& $6.17\times10^{-6}$ & 0.0061 & 0.0800 \\
&& Exact $L^\dagger$
& $5.15\times10^{-6}$ & 0.0060 & 0.0459 \\
&& Learned $\widetilde L^\dagger$
& $6.03\times10^{-6}$ & 0.0050 & 0.0470 \\
\cmidrule(lr){2-6}

& \multirow{4}{*}{500}
& PGD
& $9.10\times10^{-6}$ & 0.0080 & 0.0195 \\
&& Chebyshev
& $4.31\times10^{-6}$ & 0.0127 & 0.1626 \\
&& Exact $L^\dagger$
& $6.71\times10^{-6}$ & 0.0230 & 0.1883 \\
&& Learned $\widetilde L^\dagger$
& $2.82\times10^{-6}$ & 0.0146 & 0.1502 \\
\midrule

\multirow{8}{*}{ER}
& \multirow{4}{*}{200}
& PGD
& $4.68\times10^{-6}$ & 0.0055 & 0.0185 \\
&& Chebyshev
& $6.42\times10^{-6}$ & 0.0125 & 0.0856 \\
&& Exact $L^\dagger$
& $6.42\times10^{-6}$ & 0.0030 & 0.0230 \\
&& Learned $\widetilde L^\dagger$
& $3.75\times10^{-6}$ & 0.0040 & 0.0200 \\
\cmidrule(lr){2-6}

& \multirow{4}{*}{500}
& PGD
& $1.24\times10^{-5}$ & 0.0129 & 0.0336 \\
&& Chebyshev
& $8.88\times10^{-6}$ & 0.0150 & 0.1224 \\
&& Exact $L^\dagger$
& $8.88\times10^{-6}$ & 0.0121 & 0.1180 \\
&& Learned $\widetilde L^\dagger$
& $1.13\times10^{-5}$ & 0.0200 & 0.0773 \\

\bottomrule
\end{tabular}%
}
\end{table}
}

Table~\ref{tab:pseudoinverse-ablation} compares the four variants on grid2d and ER graphs. The learned and exact pseudoinverses achieve comparable accuracy, although their relative time-to-accuracy varies across instances. On ER with $n=200$, the learned pseudoinverse reaches the $10^{-3}$ and $10^{-6}$ thresholds in $0.0040$\,s and $0.0200$\,s, compared with $0.0030$\,s and $0.0230$\,s for the exact pseudoinverse. On grid2d with $n=500$, their final solution errors are $2.82\times10^{-6}$ and $6.71\times10^{-6}$, with $T_{10^{-6}}$ values of $0.1502$\,s and $0.1883$\,s, respectively. PGD remains faster on several individual instances, while Chebyshev generally requires longer to reach high accuracy. Overall, the learned pseudoinverse reproduces the optimization behavior of the exact $L^\dagger$ without providing uniformly faster per-instance optimization.
The computational tradeoff becomes clearer when pseudoinverse construction is included. On grid2d, DCR preprocessing increases from $1.78$\,s at $n=50$ to $28.12$\,s at $n=500$ and $1064.55$\,s at $n=2000$; the corresponding ER costs are $1.67$\,s, $26.68$\,s, and $1064.73$\,s. Once learned, the primal-reconstruction stage is considerably less expensive, requiring $4.38$\,s and $4.29$\,s at $n=2000$ for grid2d and ER, respectively.
DCR is therefore not intended to minimize end-to-end runtime for a single isolated instance. Instead, it separates graph-dependent pseudoinverse learning from instance-specific optimization. Once learned for a graph, $\widetilde{L}^{\dagger}$ can be frozen and reused across new $(A,b)$ pairs without repeating the preprocessing stage. As shown in Fig.~\ref{fig:reuse_accuracy}, this reuse preserves near-reference accuracy across independently generated instances, supporting pseudoinverse learning as reusable graph-level preprocessing.

\section{Conclusion}
\label{sec:conclusion}
This work addresses the computational challenge of the Laplacian pseudoinverse in Laplacian-regularized optimization. We propose a Difference-of-Convex Regularizer (DCR) framework that learns a reusable approximation of the Laplacian pseudoinverse through regularized Maximum Likelihood Estimation (MLE) and shrinkage-regularized Convex--Concave Procedure (CCCP) iterations. For the Laplacian-Regularized Nonnegative Least-Squares (LR-NNLS) problem, dual and KKT analysis reveals the structural role of the pseudoinverse and motivates learning a graph-dependent approximation that can be reused across problem instances. The learned pseudoinverse is then used as a graph-aware preconditioner for direct primal optimization. We further establish fixed-point existence and stability properties of the proposed learning iteration. Experiments across different graph topologies and problem scales demonstrate near-reference solution accuracy, competitive time-to-moderate-accuracy performance, and effective cross-instance reuse. Future work will extend DCR to broader graph-regularized problems and dynamic graph settings.

%\section*{Acknowledgments}

%\newpage
\bibliographystyle{IEEEtran}
\bibliography{LRMP}

%\vspace{2mm}
\appendix
\section{CCCP Formulations for Alternative Spectral Regularizers}
\label{sec:appendix-cccp-regularizers}
This appendix derives CCCP iterations for two representative spectral regularizers in Table~\ref{tab:regularizers}. As in the main text, we set $X=\Sigma^{-1}$ and formulate the objective as a DC program, with $S_k = y_k y_k^\top$. Unlike the scale-invariant regularization in \ref{subsec:phase1}, these variants alter the scaling behavior of the Tyler likelihood. To ensure well-posedness and avoid pathological scaling, we enforce the normalization $\mathrm{tr}(\Sigma)=n-1$.

\begin{figure}[!t]
    \centering
    \subfloat[Nuclear-norm]{
        \includegraphics[width=0.48\linewidth]{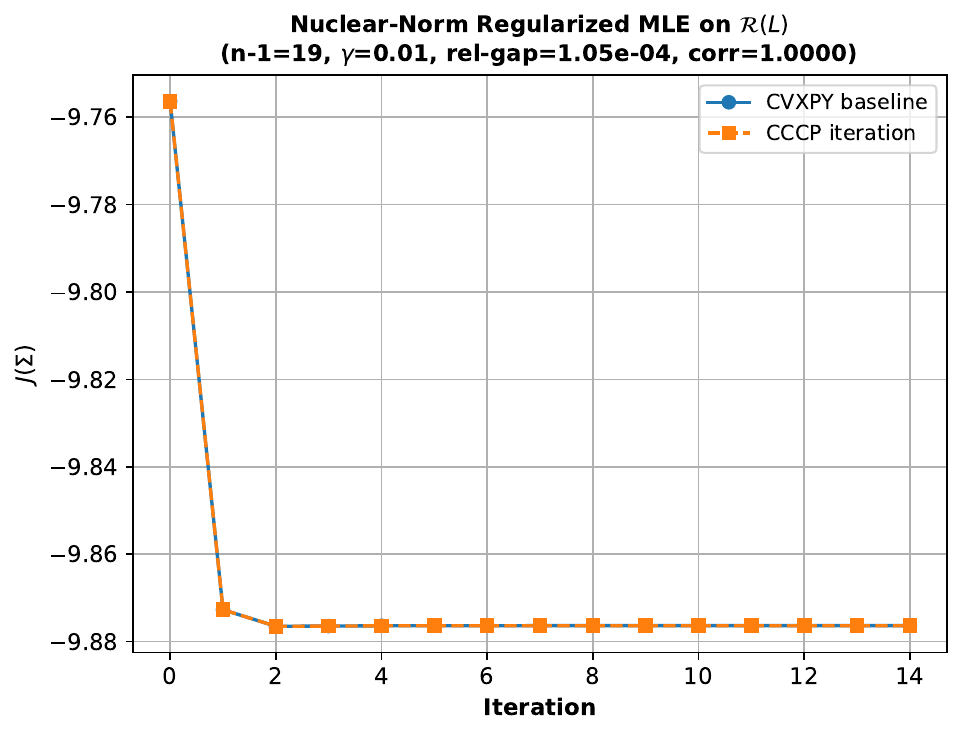}
        \label{fig:exam_appendix_nuclear_norm_mle_tyler_vs_mmcvxpy_convergence}
    }   
    \subfloat[Log-Determinant]{
        \includegraphics[width=0.48\linewidth]{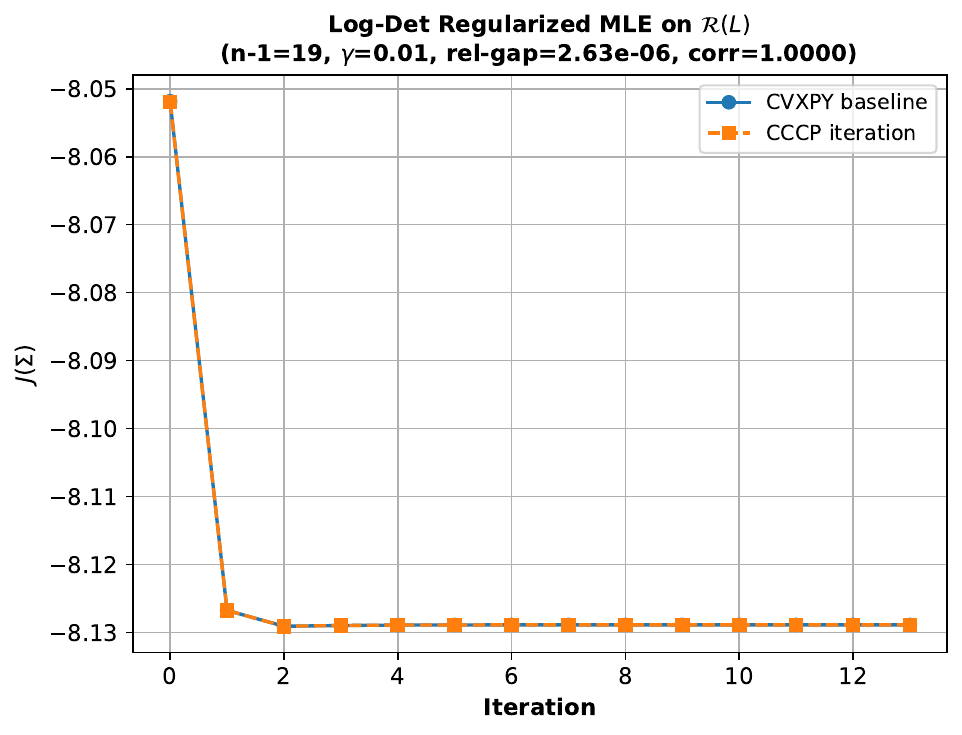}
        \label{fig:exam_appendix_log_mle_tyler_vs_mmcvxpy_convergence}
    }
    \caption{Convergence comparison between the CVXPY baseline and the fixed-point iterations \eqref{eq:appendix-nuclear-iteration} and \eqref{eq:appendix-logdet-iteration} for \eqref{prob:appendix-nuclear-sigma-mle} and \eqref{prob:appendix-logdet-sigma-mle} on $\mathcal{R}(L)$. Nearly identical objective trajectories confirm consistency with the CCCP formulation.}
    \label{fig:exam_appendix_nuclear_log_mle_tyler_vs_mmcvxpy_convergence}
\end{figure}

%\vspace{-4mm}
\subsection{Nuclear-Norm Regularization: $\|\Sigma^{-1/2}\|_*$}
\label{subsec:cccp-nuclear}
Consider the following nuclear-norm regularized MLE:
\begin{align}
\label{prob:appendix-nuclear-sigma-mle}
\min_{\Sigma\succ0}\;
J(\Sigma)
&=
\log\det(\Sigma)
+
\frac{n-1}{K}\sum_{k=1}^K
\log\!\big(y_k^\top \Sigma^{-1} y_k\big) \notag \\
&+
\gamma\,\|\Sigma^{-1/2}\|_*.
\end{align}
Here, the regularizer $\gamma\,\|\Sigma^{-1/2}\|_*$ penalizes small eigenvalues of $\Sigma$ while inducing a milder spectral shrinkage compared with $\mathrm{tr}(\Sigma^{-1})$. Note that for positive definite matrices $\Sigma$:
\begin{align}
\|\Sigma^{-1/2}\|_* 
=\mathrm{tr}\!(\sqrt{(\Sigma^{-1/2})^{\top}(\Sigma^{-1/2})})
= \sum_i \lambda_i^{-1/2}.
\end{align}
Then the problem \eqref{prob:appendix-nuclear-sigma-mle} becomes:
\begin{align}
\label{eq:appendix-nuclear-sigma}
\min_{\Sigma\succ0}\;
J(\Sigma)
=
&\log\det(\Sigma)
+
\frac{n-1}{K}\sum_{k=1}^K
\log\!\big(y_k^\top \Sigma^{-1} y_k\big) \notag \\
&+
\gamma\,\mathrm{tr}(\Sigma^{-1/2}).
\end{align}

%\vspace{1mm}
\textit{CCCP Iteration.}
Let $X=\Sigma^{-1}\succ0$. Since $\mathrm{tr}(\Sigma^{-1/2}) = \mathrm{tr}(X^{1/2})$, the objective becomes:
\begin{align}
\label{eq:appendix-nuclear-x}
F(X)
=
-\log\det(X)
+
\frac{n-1}{K}\sum_{k=1}^K
\log\!\big(\mathrm{tr}(X S_k)\big)
+
\gamma\,\mathrm{tr}(X^{1/2}).
\end{align}
% \begin{align}
% \label{eq:appendix-nuclear-x}
% F(X)
% =
% &-\log\det(X)
% +
% \frac{n-1}{K}\sum_{k=1}^K
% \log\!\big(\mathrm{tr}(X S_k)\big) \notag \\
% &+
% \gamma\,\mathrm{tr}(X^{1/2}).
% \end{align}

Define $f(X) = -\log\det(X)$ and $g(X) = \frac{n-1}{K}\sum_{k=1}^K \log(\mathrm{tr}(X S_k))+\gamma\,\mathrm{tr}(X^{1/2})$,
% \begin{align}
% g(X) = \frac{n-1}{K}\sum_{k=1}^K \log(\mathrm{tr}(X S_k))+\gamma\,\mathrm{tr}(X^{1/2}),
% \end{align}
where $f(X)$ is convex on $X\succ0$, while $g(X)$ is concave. Linearizing the concave function $g(X)$ at $X_t$ gives:
\begin{align}
\label{eq:appendix-nuclear-subproblem}
X_{t+1}
\in
\arg\min_{X\succ0}
f(X) + \langle \nabla g(X_t), X \rangle .
\end{align}
The gradients are $\nabla f(X) = -X^{-1}$ and:
\begin{align}
\nabla g(X) =
\frac{n-1}{K}\sum_{k=1}^K
\frac{S_k}{\mathrm{tr}(X S_k)}
+
\frac{\gamma}{2}X^{-1/2}.
\end{align}
The optimality condition of~\eqref{eq:appendix-nuclear-subproblem} yields:
\begin{align}
X_{t+1}^{-1}
=
\frac{n-1}{K}\sum_{k=1}^K
\frac{y_k y_k^\top}{y_k^\top X_t y_k}
+
\frac{\gamma}{2}X_t^{-1/2}.
\end{align}
Using $\Sigma_t = X_t^{-1}$, we obtain:
\begin{align}
\label{eq:appendix-nuclear-iteration}
\Sigma_{t+1}
=
\frac{n-1}{K}\sum_{k=1}^K
\frac{y_k y_k^\top}{y_k^\top \Sigma_t^{-1} y_k}
+
\frac{\gamma}{2}\Sigma_t^{1/2},
\end{align}
followed by the normalization $\mathrm{tr}(\Sigma_{t+1})={n-1}$. When $\gamma=0$, the iteration reduces to the classical Tyler update.

%\vspace{1mm}
\textit{Validation.}
Fig.~\ref{fig:exam_appendix_nuclear_log_mle_tyler_vs_mmcvxpy_convergence} (a) compares CVXPY with the fixed-point iteration~\eqref{eq:appendix-nuclear-iteration} for~\eqref{prob:appendix-nuclear-sigma-mle} on $\mathcal{R}(L)$. The trajectories nearly coincide and converge within a few iterations, confirming the correctness of the CCCP update, with relative Frobenius gap $1.05\times10^{-4}$ and correlation $1.0000$.

\subsection{Log-Determinant Barrier Regularization: $-\gamma\,\log\det(\Sigma)$}
\label{subsec:cccp-logdet}
Consider the following log-determinant regularized MLE:
\begin{align}
\label{prob:appendix-logdet-sigma-mle}
\min_{\Sigma\succ0}\;
J(\Sigma)
=
&\; \log\det(\Sigma)
+
\frac{n-1}{K}\sum_{k=1}^K
\log\!\big(y_k^\top \Sigma^{-1} y_k\big)
\notag\\
&-
\gamma\,\log\det(\Sigma).
\end{align}

The regularizer $-\gamma\,\log\det(\Sigma)$ acts as a spectral barrier that prevents eigenvalues of $\Sigma$ from collapsing to zero. Since $-\log\det(\Sigma) = -\sum_i \log \lambda_i$, this penalty grows unbounded as any eigenvalue $\lambda_i \to 0$, thereby stabilizing the spectrum of $\Sigma$. Combining the log-determinant terms, problem \eqref{prob:appendix-logdet-sigma-mle} becomes:
\begin{align}
\label{eq:appendix-logdet-sigma}
\min_{\Sigma\succ0}\;
J(\Sigma)
=
(1-\gamma)\log\det(\Sigma)
+
\frac{n-1}{K}\sum_{k=1}^K
\log\!\big(y_k^\top \Sigma^{-1} y_k\big).
\end{align}

%\vspace{1mm}
\textit{CCCP Iteration.}
Let $X=\Sigma^{-1}\succ0$, and \eqref{eq:appendix-logdet-sigma} becomes:
\begin{align}
\label{eq:appendix-logdet-x}
F(X)
=
-(1-\gamma)\log\det(X)
+
\frac{n-1}{K}\sum_{k=1}^K
\log\!\big(\mathrm{tr}(X S_k)\big).
\end{align}

Define $f(X) = -(1-\gamma)\log\det(X)$ and $g(X) =\frac{n-1}{K}\sum_{k=1}^K\log(\mathrm{tr}(X S_k))$,
% \begin{align}
% g(X) =\frac{n-1}{K}\sum_{k=1}^K\log(\mathrm{tr}(X S_k)),
% \end{align}
where $f(X)$ is convex on $X\succ0$, while $g(X)$ is concave. Linearizing $g(X)$ at $X_t$ yields:
\begin{align}
\label{eq:appendix-logdet-subproblem}
X_{t+1}
\in
\arg\min_{X\succ0}
f(X) + \langle \nabla g(X_t), X \rangle .
\end{align}
The gradients are $\nabla f(X) = -(1-\gamma)X^{-1}$ and $\nabla g(X) =
\frac{n-1}{K}\sum_{k=1}^K
\frac{S_k}{\mathrm{tr}(X S_k)}$.
% \begin{align}
% \nabla g(X) =
% \frac{n-1}{K}\sum_{k=1}^K
% \frac{S_k}{\mathrm{tr}(X S_k)} .
% \end{align}
The optimality condition of~\eqref{eq:appendix-logdet-subproblem} gives:
\begin{align}
(1-\gamma)X_{t+1}^{-1}
=
\frac{n-1}{K}\sum_{k=1}^K
\frac{y_k y_k^\top}{y_k^\top X_t y_k}.
\end{align}
Using $\Sigma_t = X_t^{-1}$, we obtain:
\begin{align}
\label{eq:appendix-logdet-iteration}
\Sigma_{t+1}
=
\frac{1}{1-\gamma}
\Big(
\frac{n-1}{K}\sum_{k=1}^K
\frac{y_k y_k^\top}{y_k^\top \Sigma_t^{-1} y_k}
\Big),
\end{align}
followed by the normalization $\mathrm{tr}(\Sigma_{t+1}) = {n-1}$. When $\gamma=0$, the iteration reduces to the classical Tyler fixed-point update.

%\vspace{1mm}
\textit{Validation.}
Fig.~\ref{fig:exam_appendix_nuclear_log_mle_tyler_vs_mmcvxpy_convergence} (b) shows the convergence of CVXPY and the fixed-point iteration~\eqref{eq:appendix-logdet-iteration} for~\eqref{prob:appendix-logdet-sigma-mle} on $\mathcal{R}(L)$. The two trajectories almost overlap, confirming the correctness of the CCCP update. Both methods converge within a few iterations and yield nearly identical solutions, with a relative Frobenius gap of $2.63\times10^{-6}$ and correlation $1.0000$.

\end{document}